\documentclass[twoside,reqno,a4paper,10pt]{amsart}
\usepackage[T1]{fontenc}
\usepackage[utf8]{inputenc}
\usepackage{color}
\usepackage{verbatim,mathrsfs,hyperref,cite}
\usepackage{amstext}
\usepackage{amsthm,amsfonts}
\usepackage{stmaryrd}
\usepackage{setspace}
\usepackage{epsfig}
\usepackage{esint}

\makeatletter

\numberwithin{equation}{section}
\numberwithin{figure}{section}
\theoremstyle{plain}
\newtheorem{thm}{\protect\theoremname}[section]
  \newtheorem{rem}{\protect\remarkname}[section]
  \newtheorem{lem}{\protect\lemmaname}[section]
  \newtheorem{defn}{\protect\definitionname}[section]

\@ifundefined{definecolor}
 {\usepackage{color}}{}

\makeatletter
\newcommand{\Rmnum}[1]{\expandafter\@slowromancap\romannumeral#1@}\makeatother

\numberwithin{equation}{section}

\allowdisplaybreaks

\newcommand{\set}[1]{\left\{#1\right\}}

\newcommand{\pr}[1]{\left(#1\right)}

\newcommand{\defs}{:=}

\DeclareMathOperator*{\dist}{dist}

\newcommand{\dif}{\mathrm{d}}

\DeclareSymbolFont{lettersA}{U}{pxmia}{m}{it}
\DeclareMathSymbol{\piup}{\mathord}{lettersA}{"19}

\makeatother

\makeatother

\usepackage{babel}
  \providecommand{\definitionname}{Definition}
  \providecommand{\lemmaname}{Lemma}
  \providecommand{\remarkname}{Remark}
\providecommand{\theoremname}{Theorem}

\begin{document}
\title[Transonic shocks with large gravity]
{Transonic Shocks for 2-D Steady Euler Flows with Large Gravity in a Nozzle for Polytropic Gases}

\author{Beixiang Fang}
\author{Xin Gao}
\author{Wei Xiang}
\author{Qin Zhao}

\address{B.X. Fang: School of Mathematical Sciences, MOE-LSC, and SHL-MAC, Shanghai Jiao Tong University, Shanghai 200240, China }
\email{\texttt{bxfang@sjtu.edu.cn}}

\address{X. Gao: Department of Applied Mathematics, The Hong Kong Polytechnic University, Hong Kong, China}
\email{\texttt{xingao@polyu.edu.hk}}

\address{W. Xiang: Department of Mathematics, City University of Hong Kong, Kowloon, Hong Kong, China}
\email{\texttt{weixiang@cityu.edu.hk }}

\address{Q. Zhao: School of Mathematics and Statistics, Wuhan University of Technology, Wuhan 430070, China}
\email{\texttt{qzhao@whut.edu.cn}}


\date{\today}

\begin{abstract}
In this paper, we are concerned with the existence of transonic shock solutions for two-dimensional (2-d) steady Euler flows of polytropic gases with the vertical gravity in a horizontal nozzle under a pressure condition imposed at the exit of the nozzle.
The acceleration of the gravity $g$ is assumed to take a generic value.
We first show that the existence of special transonic shock solutions with the flow states depending only on the variable in the gravity direction can be established if and only if the Mach number of the incoming flow satisfies certain conditions.
However, the shock position of the special solutions is arbitrary in the nozzle.
We determine the shock position and establish the existence of transonic shock solution when the boundary data are small perturbations of the special shock solutions under certain conditions.
Mathematically, the perturbation problem can be formulated as a free boundary problem of a nonlinear system of hyperbolic-elliptic mixed type and composite. Key difficulties in the analysis mainly comes from the vertical gravity. Methods and techniques are developed in this paper to deal with these key difficulties.
Finally, it turns out that the vertical gravity plays a dominant role in the mechanism determining the shock position.
\end{abstract}

\maketitle


\section{Introduction}
In this paper, we are concerned with the existence of transonic shocks, especially position of the shock front for 2-d steady Euler flows with a vertical large gravity in an almost flat nozzle for polytropic gases (see Figure \ref{fig1}).
The steady flow satisfies the following equations 
\begin{align}
&\partial_x (\rho u) + \partial_y (\rho v)=0,\label{E1}\\
&\partial_x (\rho u^2 + p) + \partial_y (\rho u v)=0,\\
&\partial_x (\rho u v ) + \partial_y (\rho v^2 + p) = - \rho g,\label{E3}\\
&\partial_x \big(\rho u(\frac12 (u^2+v^2) +i) \big) + \partial_y \big(\rho v(\frac12 (u^2+v^2) +i) \big) = -\rho g v,\label{E4}
\end{align}
where $u$, $v$, $\rho$, $p$, and $i$ are the horizontal velocity, vertical velocity, density, pressure, and enthalpy respectively. The constant $g>0$ is the acceleration of gravity. Moreover, the flow is supposed to be a polytropic gas with the thermal relation
\begin{align}\label{stateP}
p=Ae^{\frac{S}{c_v}}{\rho}^{\gamma}
\qquad\mbox{and}\qquad
i = \displaystyle\frac{\gamma p}{(\gamma -1)\rho},
 \end{align}
where $S$ is the entropy, $c_v$ is the specific heat at constant volume, $\gamma>1$ is the adiabatic exponent, and $A$ is a positive constant. 
The sound speed is given by $c=\sqrt{{\gamma p}/{\rho}}$. The flow is supersonic if the Mach number $M\defs \sqrt{u^2+v^2}/c>1$, and subsonic if $M<1$.
\begin{figure}[!ht]
	\centering
	\includegraphics[width=0.45\textwidth]{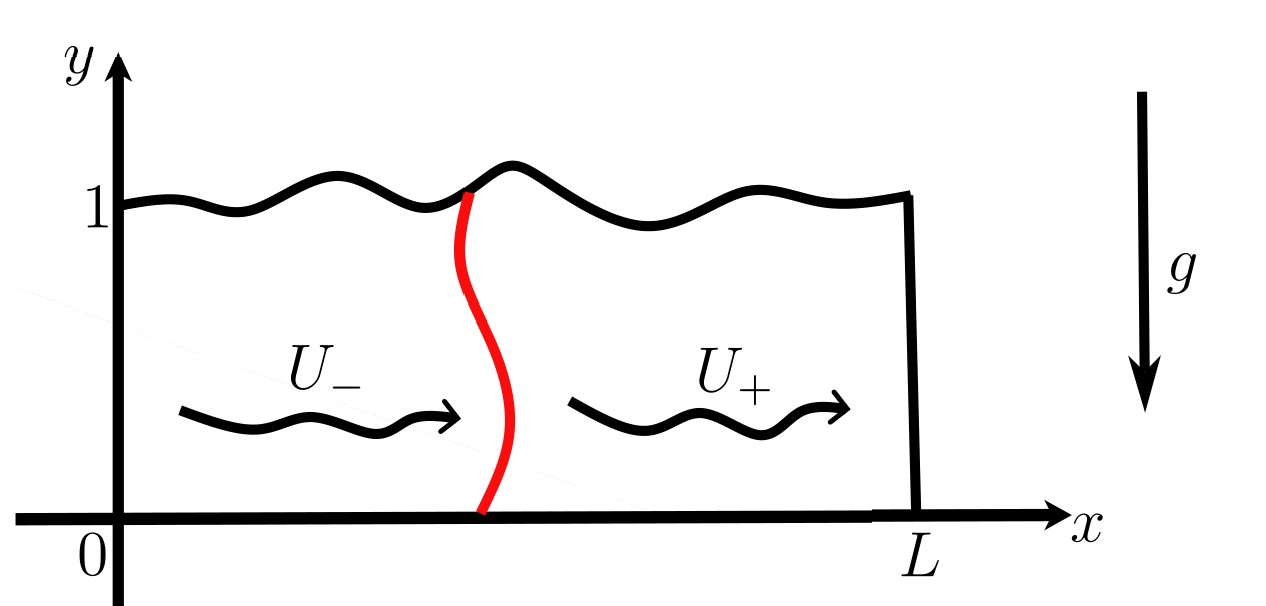}
	\caption{Transonic flow with a vertical gravity in a nozzle.\label{fig1}}
\end{figure}

In this paper, we will show the existence of the wave pattern with a single shock front in a finite nozzle for polytropic gases (see Figure \ref{fig1}) and determine the single shock location, for a generic value of $g$.
As proposed by Courant-Friedrichs in \cite{CR}, the pressure condition will be imposed at the exit of the nozzle.
In \cite{FG2022}, Fang-Gao have dealt with the problem for isothermal gases with sufficiently small gravity acceleration $g$. Under these two assumptions, the special shock solution is trivial and the zero-order terms in the Euler system can be treated simply as perturbation terms.
Then a transonic shock solution with the determined shock front position can be established when the boundary data at the entry and the exit of the nozzle are slightly perturbed from a given special solution under certain conditions.

However, for polytropic gases, special transonic shock solutions cannot be simply established as done in \cite{FG2022} . Moreover, as $ g $ takes a generic value, the zero-order terms in the Euler system \eqref{E1}-\eqref{E4} cannot be treated as perturbation terms.
Hence, the ideas in \cite{FG2022} fail to deal with the problem in this paper for polytropic gases and a generic value of $ g $.
New ideas and methods are developed to establish the existence of transonic shock solutions with a determined position of the shock front.
Different from isothermal gases, for which special shock solutions can be trivially established so that no additional conditions are needed to investigate on the Mach number of the incoming flow entering, for the polytropic gases, the special transonic shock solutions can be established if and only if the Mach number of the incoming flow ahead of the shock front satisfies certain conditions.
Such necessary and sufficient conditions will be given in Theorem \ref{BGthmxy}.

Based on the special transonic shock solutions, we will  determine the location of the shock front and establish the existence of shock solutions provided that the states at the entry, the pressure at the exit, and the nozzle walls are slightly perturbed.
Mathematically, it can be formulated as a free boundary value problem in Definition \ref{def1} with the shock front being the free boundary for a nonlinear system which changes type across the shock front: it is hyperbolic ahead of it, while it is hyperbolic-elliptic composite, with strongly coupled lower order terms, behind it. Since $g$ is no longer small, the elliptic part and the hyperbolic part of the Euler system for the subsonic flow behind the shock front are strongly coupled with the lower-order terms involved with the gravity. They cannot be simply decoupled.
Thus, new difficulties arise in solving the nonlinear free boundary problem as well as determining the position of the shock front.

Different from \cite{FGXZ2025}, gravity is an external force, which leads to a different plausible mechanism for determining the shock position. Mathematically, in this paper the $C^{0,\alpha}$-regularity for the subsonic solution behind the shock front is optimal at the intersection point between the shock front and the nozzle wall, since the nozzle wall is a small perturbation of a straight line. This low regularity leads to the non-uniqueness of the streamline for transport equations in the subsonic domain behind the shock front in Eulerian coordinates.
Therefore, different from \cite{FGXZ2025}, this problem needs to be solved in the Lagrangian coordinates. New techniques need to be developed in the Lagrangian coordinates. 

In order to establish the existence of transonic shock solutions to the free boundary problem in Lagrangian coordinates, one of the key difficulties is solving the nonlinear free boundary problem in a weighted H\"{o}lder norm space for elliptic equations with non-constant coefficients and non-local terms.
To achieve this, a specific free boundary problem for the linearized elliptic-hyperbolic composite system with non-constant coefficients and lower-order terms is formulated.
To solve this linearized problem,
one of the key points is to introduce an auxiliary problem, which provides an effective method for dealing with the lower-order terms involving with the generic value of $g$. If the linearized problem admits a solution, the boundary condition must satisfy a solvability condition, related to the large gravity and the special normal shock solution. A plausible mechanism for determining the admissible shock position is observed from the analysis of the solvability condition, which shows the large gravity plays a dominant role. Based on the above arguments, a delicate nonlinear iteration scheme is designed to the nonlinear problem.

Thanks to steady efforts made by many mathematicians, there have been plenty of results on gas flows with shocks in nozzles, for instance, see \cite{BF,GCM2007,GK2006,GM2003,GM2004,
GM2007,SC2005,Chen_S2008TAMS,SC2009,CY2008,FG2021,FG2022,FGXZ2025,FHXX2023,FLY2013,FX2021,LXY2009CMP,
LiXinYin2010PJM,LXY2013,LiuYuan2009SIAMJMA,LiuXuYuan2016AdM,
ParkRyu_2019arxiv,Park2021,PeiXiang2022,WXX2021,
XinYanYin2011ARMA,XinYin2005CPAM,XinYin2008JDE,
Yuan2012,YZ2020} and references therein. In \cite{CR}, Courant and Friedrichs first gave a systematic study of this problem from the viewpoint of nonlinear partial differential equations.
They pointed out that the position of the shock front cannot be determined unless additional conditions are imposed at the exit and the pressure condition was suggested (see \cite[P.373-P.374]{CR}). Since then, many nonlinear PDE models and different kinds of boundary conditions have been proposed to establish a rigorous mathematical analysis for this problem. This has led to the development of fruitful ideas and methods, resulting in significant progress in the field. There have a relatively systematic and complete theory for the quasi-one-dimensional model, see, for instance \cite{Liu1982ARMA,Liu1982CMP,EmbidGoodmanMajda1984,Lax,RauchXieXin2013JMPA} and references therein.
For the steady multi-dimensional case without exterior forces, two typical kinds of nozzles are studied. One is a perturbation of an expanding nozzle of an angular sector or a diverging cone, and the other one is a perturbation of a flat nozzle with two parallel walls. See \cite{SC2009,GK2006,GM2003,LXY2009CMP,LXY2013,XinYanYin2011ARMA,XinYin2005CPAM,
XinYin2008JDE,FX2021} and references therein.
See also \cite{FG2022,FSZ2023,GLY2020,YZ2020} for the stability of transonic shocks with small gravity, exothermic, heat transfer, or friction in nozzles.

The rest of the paper is organized as follows.
In Section 2, the free boundary problem will be formulated, special normal transonic shock solutions will be constructed, and the main theorem will be stated. In Section 3, the nonlinear boundary value problem in the subsonic domain behind the shock front will be reformulated.
In Section 4, a well-posedness theorem for a linearized boundary value problem will be established. This theorem plays a crucial role in solving the nonlinear boundary value problem introduced in Section 3. Based on Section 4, a nonlinear iteration scheme will be designed in Section 5, and Theorem \ref{mainthm3} will be proved.

\section{Free boundary problem and main theorem}
\subsection{Transonic shock solution}
Let $ \varphi_w(x) $ be a given function with $ \varphi_w(0) =1$ and the nozzle is the domain
\begin{equation*}
\mathcal{D} = \{ (x,y): 0 < x < L, 0 < y < \varphi_w(x) \}.
\end{equation*}
Denote the entrance, the exit, the lower boundary and the upper boundary of the nozzle by
\begin{align}
E_0 &= \{(x,y): x=0,\,y \in [0,1]\},\label{E0entrance}\\	
E_L &= \{(x,y): x=L,\, y \in [0, \varphi_w (L)] \},\\
W_0 &= \{(x,y): x \in [0,L], \, y=0 \},\\
W_1 &= \{(x,y): x \in [0,L],\, y=\varphi_w(x) \}.\label{W1Wall}
\end{align}
Let the position of the shock front be
\begin{equation}\label{shockEuler}
E_s\defs\set{(x, y)\in \mathbb{R}^2 :  x = \varphi_s(y),\, 0 < y < y_s},
\end{equation}
where $(y_s, \varphi_s (y_s))$ is the intersection of the shock-front and the upper boundary, i.e. $y_s=\varphi_w(\varphi_s(y_s))$. Then the domain $\mathcal{D}$ is divided into two parts by $E_s$:
\begin{align}
& \mathcal{D}_- = \{ (x, y)\in \mathbb{R}^2 : 0 < x < \varphi_s(y),\, 0 < y < \varphi_w(x)\},\label{mathcalD-}\\
& \mathcal{D}_+ = \{ (x, y)\in \mathbb{R}^2 : \varphi_s(y) < x < L,\, 0 < y < \varphi_w(x) \}.\label{mathcalD+}
\end{align}
$ \mathcal{D}_- $ is the region of the supersonic flow ahead of the shock front, and $ \mathcal{D}_+ $ is the region of the subsonic flow behind it.
It is well known that the following Rankine-Hugoniot (R.-H.) conditions hold on $E_s$.
\begin{align}
&[\rho u] - \varphi_s'[\rho v]=0,\label{RH1}\\
&[\rho u^2 + p] - \varphi_s'[\rho u v]=0,\\
&[\rho u v ] - \varphi_s'[\rho v^2 + p]=0,\\
&[\frac12 (u^2+v^2) + i] = 0,\label{RH4}
\end{align}
where $[\cdot]$ stands for the jump of the corresponding quantity across the shock front.

We consider a supersonic flow with a large gravity enters the nozzle $\mathcal{D}$ , and leaves it with a relatively high pressure on $E_L$. It is expected that a transonic shock front appears in $\mathcal{D}$. Let $\theta = \arctan \displaystyle\frac{v}{u}$ be the flow angle, and let $q= \sqrt{u^{2} + v^{2}}$ be the flow speed. Denote $U\defs (p,\theta,q,S)^\top$.
Hence, in this paper, we consider the existence of transonic shock solutions of the following free boundary value problem and try to determine the shock location function $\varphi_s(y)$.
\begin{defn}[Transonic shock solution]\label{def1}
For given functions $U_{-}^{0}(y) = (p_-^0, \theta_-^0, q_-^0, S_-^0)^{\top}(y)$ and $p_{+}^{L}(y)$,
$ \pr{ U_-(x,y);\ {U}_+(x,y);\ \varphi_s(y)}$ is a transonic shock solution of the following free boundary value problem:
\begin{enumerate}
\item $U_-(x,y) = (p_-,\theta_-,q_-,S_-)^\top(x,y)$ satisfies the Euler equations \eqref{E1}-\eqref{E4} in $\mathcal{D}_-$, and the boundary conditions
\begin{align}\label{eq11}
U_-=& U_{-}^{0}(y) \qquad &\text{on }&  E_0,\\
\theta_- =& 0 \qquad &\text{on }& W_0\cap \overline{\mathcal{D}_-},\label{eq11-}\\
\tan \theta_{-} =& \varphi_w'(x) \qquad &\text{on }& W_1\cap \overline{\mathcal{D}_-};
\end{align}

\item $U_+(x,y)=(p_+,\theta_+,q_+,S_+)^\top(x,y) $ satisfies the Euler equations \eqref{E1}-\eqref{E4} in $ \mathcal{D}_+$, and the boundary conditions
\begin{align}\label{eq14}
\theta_+ =& 0 \qquad &\text{on } & W_0\cap \overline{\mathcal{D}_+},\\
\tan \theta_{+}  =& \varphi_w'(x) \qquad &\text{on } & W_1\cap \overline{\mathcal{D}_+},\\
\label{eq14+}
p_+ =&p_{+}^{L}(y)\qquad &\text{on } & E_L;
\end{align}
	
\item $(U_-; U_+; \varphi_s)$ satisfies the R.-H. conditions \eqref{RH1}-\eqref{RH4} on $E_s $.
\end{enumerate}
\end{defn}

\subsection{Special normal transonic shock solutions}
The objective of this paper is to find a transonic shock solution in the sense of Definition \ref{def1} and determine the shock position. To achieve this, one of the key points is to find special normal shock solutions with large gravity for polytropic gas and $\varphi_w(x) \equiv 1$.
Let  $\mathcal{D}^{\bar{x}_s}=\{ (x,y): 0 < x < L, 0 < y < 1\}$. Let $\bar{x}_s \in (0,L)$ be an arbitrary normal shock location, the domain $\mathcal{D}^{\bar{x}_s}$ is then divided into two parts: the supersonic domain $\mathcal{D}_-^{\bar{x}_s}$ and the subsonic domain $\mathcal{D}_+^{\bar{x}_s}$:
\begin{align}
&\mathcal{D}_-^{\bar{x}_s} \defs \{ (x, y)\in \mathbb{R}^2 : 0 < x < \bar{x}_s,\, 0 < y < 1\},\\
&\mathcal{D}_+^{\bar{x}_s} \defs \{ (x, y)\in \mathbb{R}^2 : \bar{x}_s < x < L,\, 0 < y < 1\}.
\end{align}
In this subsection, we will construct special normal transonic shock solutions with new observations
of the form 
\begin{align}
&\overline{U}_-(y) = (\bar{p}_-, \bar{\theta}_-, \bar{q}_-, \bar{S}_-)^\top(y) \defs
(\bar{p}_-(y), 0, \bar{q}_-(y), \bar{S}_-(y))^\top\qquad\mbox{in }\mathcal{D}_-^{\bar{x}_s} ,\label{BS--thm}\\
&\overline{U}_+(y) =  (\bar{p}_+, \bar{\theta}_+, \bar{q}_+, \bar{S}_+)^\top (y)\defs (\bar{p}_+(y), 0, \bar{q}_+(y), \bar{S}_+(y))^\top\qquad\mbox{in }\mathcal{D}_+^{\bar{x}_s}.\label{BS++thm}
\end{align}
The existence of such special normal transonic shock solutions in the form of \eqref{BS--thm}-\eqref{BS++thm} can be established by proving the following theorem.
\begin{thm}\label{BGthmxy}
For a given $C^3$-function $\bar{q}_-(y)>0$ in $\mathcal{D}_-^{\bar{x}_s}$, the special normal transonic shock solutions of the form $(\overline{U}_- (y);\ \overline{U}_+ (y);\ \bar{x}_s)$ as described in \eqref{BS--thm}-\eqref{BS++thm} with $\overline{M}_-(1)>1$ and $\bar{p}_-(1)>0$, satisfy \eqref{E1}-\eqref{E4} and \eqref{RH1}-\eqref{RH4}, if and only if $\overline{M}_-^2$ satisfies
\begin{align}\label{BS+8}
\frac{d}{dy}\big(\frac{1}{\overline{M}_-^2(y)}\big) = -\frac{g}{\bar{q}_-^2(y)} \frac{\gamma-1}{4}\big(1-\gamma
-\frac{2}{\overline{M}_-^2(y)}+\frac{(\gamma+1)^2\overline{M}_-^2(y)}{2+(\gamma-1)\overline{M}_-^2(y)}
\big)
\end{align}
and
\begin{equation}\label{eq:2.23}
\frac{\dif\bar{p}_-(y)}{\dif y} = -\frac{\gamma \bar{p}_-(y)\overline{M}_-^2(y)}{\bar{q}_-^2(y)}g,
\end{equation}
where $\overline{M}_-^2(y)=\frac{\bar{q}_-^2(y)\bar{\rho}_-(y)}{\gamma \bar{p}_-(y)}$.
Then such solutions $(\overline{U}_- (y);\ \overline{U}_+ (y);\ \bar{x}_s)$ exist.
\end{thm}

\begin{proof}
This proof is divided into two steps.

\emph{Step 1:}
For the solutions of the form $(\overline{U}_- (y);\ \overline{U}_+ (y);\ \bar{x}_s)$ as described in \eqref{BS--thm}-\eqref{BS++thm}, they satisfy \eqref{E1}-\eqref{E4} and $\eqref{RH1}$-$\eqref{RH4}$ if and only if they satisfy
\begin{align}
\frac{\dif\bar{p}_-(y)}{\dif y} =& -\bar{\rho}_-(y)g,
\quad \text{in} \quad \mathcal{D}_-^{\bar{x}_s},
\label{BS-}\\
\frac{\dif\bar{p}_+(y)}{\dif y} =& -\bar{\rho}_+(y)g,
\quad \text{in} \quad \mathcal{D}_+^{\bar{x}_s},\label{BS+}\\
\bar{p}_+=&((1+{\mu}^2)\overline{M}_-^2-{\mu}^2)\bar{p}_-,\label{RH5}\\
\bar{q}_+  =& {\mu}^2(\bar{q}_- + \frac{2}{\gamma-1}\frac{\bar{c}_{-}^2}{\bar{q}_-}),\label{RH6}\\
\bar{\rho}_+ =& \frac{{\bar{\rho}_- \bar{q}_-}}{\bar{q}_+}\label{RH7},
\end{align}
where $\bar{c}_{-}= \sqrt{\displaystyle\frac{\gamma \bar{p}_-}{\bar{\rho}_-}}$, $\overline{M}_- = \displaystyle\frac{\bar{q}_-}{\bar{c}_{-}}>1$, and $ {\mu}^2=\displaystyle\frac{\gamma-1}{\gamma+1}$.

It is easy to see that \eqref{eq:2.23} is equivalent to \eqref{BS-}.

Next, by applying \eqref{RH5}-\eqref{RH7}, further calculations yield that
  \begin{align}\label{BS+1}
   \bar{p}_+=\frac{2\bar{\rho}_-\bar{q}_-^2}{\gamma+1}   - \frac{\gamma-1}{\gamma+1}\bar{p}_-,\quad
  \bar{\rho}_+ = \frac{(\gamma+1)\bar{\rho}_- \overline{M}_-^2}{(\gamma-1) (\overline{M}_-^2 + \frac{2}{\gamma-1})}.
 \end{align}
Employing \eqref{BS-}-\eqref{BS+}, then \eqref{BS+1} yields that
 \begin{align}\label{rhoqq}
 \frac{\dif(\bar{\rho}_-\bar{q}_-^2)}{\dif y} = -\frac{\gamma-1}{2}\big( 1 +  \frac{(\gamma+1)^2}{(\gamma -1)^2}\frac{1}{ 1 + \frac{2}{(\gamma-1)\overline{M}_-^2}} \big)\bar{\rho}_- g.
 \end{align}
 Note
 \begin{align}\label{frac1M2}
   \frac{\dif}{\dif y}\big(\frac{1}{\overline{M}_-^2(y)}\big) =  \frac{\dif}{\dif y}\big(\frac{\gamma \bar{p}_-}{\bar{\rho}_- \bar{q}_-^2}\big) = \frac{\gamma }{\bar{\rho}_- \bar{q}_-^2} \frac{\dif\bar{p}_-}{\dif y} - \frac{\gamma \bar{p}_-}{\bar{\rho}_-^2 \bar{q}_-^4}  \frac{\dif (\bar{\rho}_-\bar{q}_-^2)}{\dif y}.
 \end{align}
So \eqref{BS+8} follows from \eqref{BS-}, \eqref{rhoqq}, and \eqref{frac1M2}.

On the other hand, by
 \eqref{BS+1}, one has
 \begin{align}
&\text{left hand side of }  \eqref{BS+} = \frac{\dif\bar{p}_+}{\dif y} = \frac{2}{\gamma+1} \frac{\dif(\bar{\rho}_-\bar{q}_-^2)}{\dif y}- \frac{\gamma-1}{\gamma+1}\frac{\dif \bar{p}_-}{\dif y},\label{LHS=}\\
&\text{right hand side of }  \eqref{BS+} = - \bar{\rho}_+ g = -\frac{(\gamma+1)\bar{\rho}_- \overline{M}_-^2}{(\gamma-1) (\overline{M}_-^2 + \frac{2}{\gamma-1})} g.\label{RHS=}
 \end{align}
 By \eqref{frac1M2}, one has
 \begin{align}\label{rhoq2M}
   \frac{\dif(\bar{\rho}_-\bar{q}_-^2) }{\dif y}= \frac{\bar{\rho}_-\bar{q}_-^2}{\bar{p}_-} \frac{\dif\bar{p}_-}{\dif y} - \frac{\bar{\rho}_-^2 \bar{q}_-^4}{\gamma \bar{p}_-} \frac{\dif}{\dif y}\big(\frac{1}{\overline{M}_-^2}\big).
 \end{align}
Substituting \eqref{BS+8} into \eqref{rhoq2M} and applying \eqref{BS-}, \eqref{LHS=}, \eqref{RHS=}, it follows that
\eqref{BS+} holds if \eqref{BS+8} holds.

\emph{Step 2:}
In this step, we prove the existence of solutions of the form $(\overline{U}_- (y);\ \overline{U}_+ (y);\ \bar{x}_s)$.

Let $t(y)\defs  \frac{1}{\overline{M}_-^2(y)}$, then \eqref{BS+8} becomes:
\begin{align}\label{dtyeq}
\frac{\dif t(y) }{\dif y} =& -\frac{g}{\bar{q}_-^2(y)} \frac{\gamma-1}{4}\big(1-\gamma
-2t(y) +\frac{(\gamma+1)^2}{\gamma -1 + 2t(y)}
\big)\notag\\
=& -\frac{g}{\bar{q}_-^2(y)} \frac{(\gamma-1)^2}{4}\big(   -\big(1+\frac{2t(y)}{\gamma-1} \big)+\frac{1}{\mu^4}\frac{1}{1+\frac{2t(y)}{\gamma-1}} \big).
\end{align}
Let $T(y) \defs 1+\frac{2t(y)}{\gamma-1}$. If $t(y)\in(0,1)$ for $y\in [0,1]$, then
\begin{align}
  -\big(1+\frac{2t(y)}{\gamma-1} \big)+\frac{1}{\mu^4}\frac{1}{1+\frac{2t(y)}{\gamma-1}} =  \frac{(1- \mu^2 T(y))(1+\mu^2 T(y))}{\mu^4 T(y)} >0,
\end{align}
Thus, \eqref{dtyeq} yields that for $y\in [0, 1]$,
$\frac{\dif t(y) }{\dif y} < 0$ for $g>0$.
So
\begin{align}
t(y) > t(1) >0, \quad \text{for}\quad g>0.
\end{align}
Applying \eqref{dtyeq}, further calculations yield that
\begin{align}\label{t(y)}
  t(y) = \frac{\gamma+1}{2} \sqrt{ 1 - \big( 1 - \mu^4\big(  1+\frac{2t(1)}{\gamma-1} \big)^2 \big) e^{-(\gamma-1)g\int_y^1  \frac{\dif\tau}{\bar{q}_-^2(\tau)}}}-\frac{\gamma-1}{2}.
\end{align}
Employing $t(1) =  \frac{1}{\overline{M}_-^2(1)} \in (0,1)$, \eqref{t(y)} yields that
\begin{align}
  t(y) < 1, \quad \text{for}\quad y\in[0,1], \,\, g>0.
\end{align}
By \eqref{eq:2.23},
\begin{align}\label{barp-y}
\bar{p}_- (y) = \bar{p}_-(1) e^{\gamma g\int_y^{1} \frac{1}{t(\tau)\bar{q}_-^2(\tau)} \dif\tau}.
\end{align}
Substituting \eqref{barp-y} into the formula
\begin{align}\label{barrhog-}
 \bar{\rho}_-(y) = \frac{\gamma \bar{p}_-(y)}{t(y) \bar{q}_-^2(y)},
\end{align}
we have $\bar{\rho}_-(y)$. Then, by the state equation \eqref{stateP}, we find the expression of $\overline{S}_-(y)$. Finally, the expression of $\overline{U}_+(y)$ can be derived from \eqref{RH5}-\eqref{RH7}.
\end{proof}
\vspace{-0.25cm}
\begin{rem}
Following the proof of Theorem \ref{BGthmxy}, it is easy to see that for a given $C^3$-function $\bar{q}_-(y)>0$ and a constant $g<0$ with conditions $\overline{M}_-(0)>1$ and $\bar{p}_-(1)>0$,
there exist unique functions $\overline{U}_-(y)$ in $\mathcal{D}_-^{\bar{x}_s}$
satisfying \eqref{BS+8} and \eqref{eq:2.23}.
\end{rem}
\subsection{Main theorem}
Based on the constructed special normal transonic shock solutions, we consider the following free boundary problem in the sense of Definition \ref{def1}.
Let
\begin{align}
U_{-}^{0}(y) =& \overline{U}_-(y) +  ( \sigma p_{en}^{\sharp}(y), 0, 0,0)^{\top},\\
  \varphi_w (x) =& 1 + \int_0^{x} \tan (\sigma\Theta(s)) \dif s,\\
  p_{+}^{L}(y) =& \bar{p}_+ (y) + \sigma p_{ex}(y),
\end{align}
where $\sigma>0$ is a small constant, and $p_{en}^{\sharp} (y), \Theta(x), p_{ex}(y)$ are given functions satisfying the following conditions:
\begin{align}
&p_{en}^{\sharp}(y) \in C^{2,\alpha}([0,1]), \quad  p_{ex}(y)\in C^{2,\alpha}([0,\varphi_w(L)]), \quad \Theta(x)\in C^{2,\alpha}([0,L]),\label{assumexy}\\
&\Theta(0)=\Theta'(0) = \Theta''(0) =0, \\
&\sigma \frac{\dif}{\dif y} p_{en}^{\sharp}(y) = -g\big( \big(\frac{\bar{p}_- + \sigma p_{en}^{\sharp}}{A e^{\frac{\bar{S}_-}{c_v}}}\big)^{\frac{1}{\gamma}} - \bar{\rho}_-\big)(y),\,\, \text{at} \,\, y=0,1. \label{CC01=}
\end{align}
Let
\begin{align}
\wp_1\defs  \frac{\gamma-1+2t(1)}{\gamma+1}\min\limits_{y\in[0,1]} \bar{q}_-(y),\quad \wp_2\defs \frac{\gamma \bar{p}_-(1)}{ t(0)\max\limits_{y\in[0,1]}\bar{q}_-(y)},
\end{align}
where $t(1) =  \frac{1}{\overline{M}_-^2(1)} \in (0,1)$ and $\bar{p}_-(1)>0$ are given constants. Let
\begin{align}
\bar{p}_- (0) =& \bar{p}_-(1) e^{\gamma g\int_0^1 \frac{1}{t(y)\bar{q}_-^2(y)} \dif y},\\
 t(0) =& \frac{\gamma+1}{2} \sqrt{ 1 - \Big( 1 - \mu^4\big(  1+\frac{2t(1)}{\gamma-1} \big)^2 \Big) e^{- (\gamma-1)g\int_0^1  \frac{\dif y}{\bar{q}_-^2(y)}}}-\frac{\gamma-1}{2}.
\end{align}
In this paper, we will establish the following theorem.
\vspace{-0.35cm}
\begin{thm}\label{mainthm3}
  Assume that \eqref{assumexy}-\eqref{CC01=} and the assumptions \eqref{assump1}-\eqref{eq:<J2<} in Lemma \ref{determinemathcalK0} hold. Let
\begin{align}\label{assumpthm}
0< L \leq \min\Big\{\sqrt{\frac{\alpha(1-\alpha)}{2}}\frac{1}{g} \wp_1^2,\,\, \frac{\sqrt{2}\wp_1^2}{2g} \min\{1, \wp_2\}\Big\}.
\end{align}
Then there exists a sufficiently small positive constant $\sigma_0$, which depends only on $\overline{U}_{\pm}$ and $L$.
For any $0< \sigma \leq \sigma_0$, there exists a solution $(U_-, U_+ ; \psi)$ for the free boundary problem in the sense of Definition \ref{def1} satisfying the following estimates:
\begin{align}
&\| U_- - \overline{U}_- \|_{C^{2,\alpha}(\mathcal{D}_-)}  \leq C_-\sigma,\\
&\| (p_+ - \bar{p}_+, \theta_+ )\circ \pounds_{\mathcal{K}_0}^{-1} \|_{1,\alpha;\Omega_+^{\mathcal{K}_0}}^{(-\alpha;\{Q_i\})}
+ \| (q_+ - \bar{q}_+, S_+ - \bar{S}_+ )\circ \pounds_{\mathcal{K}_0}^{-1} \|_{1,\alpha;\Omega_+^{\mathcal{K}_0}}^{(-\alpha;\{\overline{\Gamma_0^{\mathcal{K}_0}}\cup \overline{\Gamma_{\overline{m}}^{\mathcal{K}_0}}\})}\notag\\
& +\| {\psi}'\|_{1,\alpha;\Gamma_s^{\mathcal{K}_0}}^{(-\alpha;\{Q_1,Q_4\})}
+ |\psi(\overline{m}) - \mathcal{K}_0| \leq C_+ \sigma,\label{U+esthm}
\end{align}
where the positive constant $C_-$ depends on $\overline{U}_-$, $L$ and $\alpha$, the positive constant $C_+$ depends on $\overline{U}_{\pm}$, $L$ and $\alpha$. $\mathcal{K}_0$ is the approximate shock position, which is obtained in Lemma \ref{determinemathcalK0}.
In addition, $\pounds_{\mathcal{K}_0}^{-1}$, $\Omega_+^{\mathcal{K}_0}$, $Q_i,(i=1,2,3,4)$ and $\Gamma_0^{\mathcal{K}_0}, \Gamma_{\overline{m}}^{\mathcal{K}_0}, \Gamma_s^{\mathcal{K}_0}$ are defined in \eqref{poundsK=}, \eqref{Omega+defs}, \eqref{Qidefs=}-\eqref{Gammadefs=} and \eqref{GammasK0=} with $\mathcal{K}\defs \mathcal{K}_0$.
\end{thm}

\begin{rem}
The assumptions \eqref{assump1}-\eqref{eq:<J2<} are sufficient conditions for determining of the approximate shock position $\mathcal{K}_0$, and can also be replaced by the assumptions \eqref{assump1rem}-\eqref{<J2<eq}. See Lemma \ref{determinemathcalK0} and Lemma \ref{2suffxi0=} for more details.
\end{rem}
\begin{rem}[Definition of the weighted H\"{o}lder norms]
Let $\Gamma$ be an open portion of $\partial\Omega$, for any $\mathbf{x}$, $\mathbf{y}$ $\in$ $\Omega$, define
\begin{equation}
d_{\mathbf{x}}: = \dist(\mathbf{x}, \Gamma),\,\,\, \text{and} \,\,\,
d_{\mathbf{x},\mathbf{y}}:= \min (d_{\mathbf{x}},d_{\mathbf{y}}).
\end{equation}
Let $\alpha\in(0,1)$ and $\delta\in \mathbb{R}$, we define:
\begin{align}
&[u]_{k ,0;\Omega}^{(\delta;\Gamma)} : = \sup\limits_{\mathbf{x}\in\Omega, |\mathbf{m}| = k}\big(d_{\mathbf{x}}^{\max(k + \delta,0)}|D^{\mathbf{m}} u(\mathbf{x})| \big),\\
&[u]_{k,\alpha;\Omega}^{(\delta;\Gamma)} : = \sup\limits_{{\mathbf{x},\mathbf{y}}\in\Omega, \mathbf{x}\neq\mathbf{y}, |\mathbf{m}| = k}\big(d_{\mathbf{x},\mathbf{y}}^{\max( k +\alpha+\delta,0)}\displaystyle\frac
{|D^{\mathbf{m}} u(\mathbf{x})- D^{\mathbf{m}} u(\mathbf{y})|}{|\mathbf{x} - \mathbf{y}|^\alpha}\big),\\
& \| u \|_{k,\alpha;\Omega}^{(\delta;\Gamma)} :  =\sum_{i =0}^{k} [u]_{i,0;\Omega}^{(\delta;\Gamma)} + [u]_{k,\alpha;\Omega}^{(\delta;\Gamma)}.
\end{align}
\end{rem}

\vspace{-0.65cm}
\section{Reformulation of the free boundary value problem behind the shock front}
To prove Theorem \ref{mainthm3}, it suffices to establish the well-posedness of the subsonic solution behind the shock front for the nonlinear free boundary value problem.
In this section, we will introduce coordinates transformations to straighten the nozzle wall and to fix the shock front. Based on this, we will propose a reformulated problem in the subsonic domain.
\subsection{Flatten the nozzle wall and fix the shock front}
To flatten the nozzle wall, we introduce the following Lagrangian coordinate:
\begin{align}\label{Lagrangexieta=}
({\xi}, {\eta}) = (x, \frac{\overline{m}}{m} \int_{(0,0)}^{(x,y)}\rho u(s,t) \dif t - \rho v(s,t) \dif s),
\end{align}
where $\overline{m} \defs\int_{0}^{1}\bar{\rho}_-(y)\bar{q}_-(y) \dif y$ and $m\defs\int_{0}^{1}\big( 1+\frac{\sigma p_{en}^{\sharp}(y)}{\bar{p}_-(y)}\big)^{\frac{1}{\gamma}}\bar{\rho}_-(y)\bar{q}_-(y) \dif y$. In the coordinate \eqref{Lagrangexieta=}, the shock front
\eqref{shockEuler} becomes
\begin{align}\label{Gammas==}
\Gamma_s\defs\set{(\xi, \eta)\in \mathbb{R}^2 :  \xi =\psi(\eta),\, 0 < \eta < \overline{m}}.
\end{align}
To fix the shock front $\Gamma_s$ in \eqref{Gammas==}, we introduce the following coordinate transformation: Let $\mathcal{K}\in(0,L)$ be a constant and
\begin{align}\label{tildexieta}
\pounds_\mathcal{K}: (\tilde{\xi}, \tilde{\eta}) = \big( L + \frac{L - \mathcal{K}}{L - {\psi}(\eta)}(\xi - L), \,\eta \big),
\end{align}
with the inverse
\begin{align}\label{poundsK=}
 \pounds_{\mathcal{K}}^{-1}: (\xi,\, \eta) = \big(L + \frac{L -  {\psi}(\tilde{\eta})}{L - \mathcal{K}}(\tilde{\xi} - L),\,
\tilde{\eta}\big).
\end{align}
Under the transformations \eqref{Lagrangexieta=} and \eqref{tildexieta}, equations \eqref{E1}-\eqref{E4} become
\begin{align}
&\frac{L - \mathcal{K}}{L - {\psi}(\tilde{\eta})}\partial_{\tilde{\xi}} \big(\frac{1}{\tilde{\rho} \tilde{q} \cos\tilde{\theta}}\big) - \frac{m}{\overline{m}} \big(\frac{(\tilde{\xi} -L)\psi'(\tilde{\eta})}{L - \psi(\tilde{\eta})} \partial_{\tilde{\xi}} + \partial_{\tilde{\eta} }\big)(\tan\tilde{\theta}) = 0,\label{LCmass}\\
&\frac{L - \mathcal{K}}{L - {\psi}(\tilde{\eta})}\partial_{\tilde{\xi}}  (\tilde{q} \sin\tilde{\theta}) + \frac{m}{\overline{m}} \big(\frac{(\tilde{\xi} -L)\psi'(\tilde{\eta})}{L - \psi(\tilde{\eta})} \partial_{\tilde{\xi}} + \partial_{\tilde{\eta} }\big) \tilde{p} + \frac{g}{\tilde{q}\cos\tilde{\theta}} =
0,\label{LCmo}\\
&\frac{L - \mathcal{K}}{L - {\psi}(\tilde{\eta})}\partial_{\tilde{\xi}} \big(\frac12 \tilde{q}^2 + \tilde{i} \big)+ g\tan\tilde{\theta} =0,\label{LCB}
\\
& \partial_{\tilde{\xi}} \tilde{S} = 0,\label{LCS}
\end{align}
where $(\tilde{p}, \tilde{\theta}, \tilde{q}, \tilde{S})(\tilde{\xi}, \tilde{\eta})\defs (p,\theta,q,S) \circ \pounds_{\mathcal{K}}^{-1} (\tilde{\xi}, \tilde{\eta})$.
In addition, 
the background solutions $\overline{U}_{\pm}(y)$ become the vector function of variable $\tilde{\eta}$, \emph{i.e.},
$\overline{U}_-(\tilde{\eta})\defs (\bar{p}_-(\tilde{\eta}),0 ,\bar{q}_-(\tilde{\eta}),\bar{S}_-(\tilde{\eta}))^\top$ and $\overline{U}_+(\tilde{\eta})\defs (\bar{p}_+(\tilde{\eta}),0 ,\bar{q}_+(\tilde{\eta}),\bar{S}_+(\tilde{\eta}))^\top$
with satisfying
\begin{align}\label{L-+}
\frac{\dif\bar{p}_-(\tilde{\eta})}{\dif \tilde{\eta}} = -\frac{g}{\bar{q}_-(\tilde{\eta})},\quad
\frac{\dif\bar{p}_+(\tilde{\eta})}{\dif \tilde{\eta}} = -\frac{g}{\bar{q}_+(\tilde{\eta})}.
\end{align}
The given quantity $p_{en}^{\sharp}(y)$ becomes
\begin{align}\label{penetadef}
p_{en}(\tilde{\eta}) \defs p_{en}^{\sharp}(y(0,\tilde{\eta})) = p_{en}^{\sharp}\big( \frac{m}{\overline{m}} \int_0^{\eta}\frac{1}{\rho u(0,t)}\dif t  \big).
\end{align}
The domains \eqref{mathcalD-}-\eqref{mathcalD+} and boundaries \eqref{E0entrance}-\eqref{W1Wall} become
\begin{align}
&\Omega_-^{\mathcal{K}} \defs  \{(\tilde{\xi},\tilde{\eta}): 0<\tilde{\xi} < \mathcal{K},\,\, 0<\tilde{\eta} < \overline{m}\},\label{Omega-defs}\\
& {\Omega}_+^{\mathcal{K}} = \{ (\tilde{\xi}, \tilde{\eta})\in \mathbb{R}^2 : \mathcal{K} < \tilde{\xi} < L,\,\,  0 < \tilde{\eta} < \overline{m} \},\label{Omega+defs}\\
&\Gamma_{en} = \{(\tilde{\xi},\tilde{\eta}): \tilde{\xi}=0, \, 0<\tilde{\eta}< \overline{m} \}, \\
&\Gamma_{ex} = \{(\tilde{\xi},\tilde{\eta}): \tilde{\xi}=L,\, 0<\tilde{\eta}< \overline{m} \},\\
&\Gamma_{0} = \{(\tilde{\xi},\tilde{\eta}): 0<\tilde{\xi}<L,\, \tilde{\eta}=0 \}, \\
&\Gamma_{\overline{m}} = \{(\tilde{\xi},\tilde{\eta}): 0<\tilde{\xi}<L, \,\tilde{\eta}= \overline{m} \},\\
&\Gamma_{s}^{\mathcal{K}} = \{ (\tilde{\xi},\tilde{\eta})\in \mathbb{R}^2 : \tilde{\xi} = \mathcal{K},\,\, 0 <\tilde{\eta} < \overline{m} \}.\label{GammasK0=}
\end{align}
For notational simplicity, we drop `` $\widetilde{}$ '' in the following arguments.

\vspace{-0.20cm}
\subsection{Reformulated problem in the subsonic domain}
The unique existence of the supersonic solution $U_-$ ahead of the shock front can be easily established by employing the method of characteristic line (cf. Section 2 - Section 3 of the book \cite{LY1985}).
Therefore, to show  Theorem \ref{mainthm3}, we only need to consider the solution in the subsonic domain.
Let $U_+(\xi,\eta) =\overline{U}_+(\eta)  + \delta U_+(\xi,\eta)$ be the state of the subsonic flow behind the shock front, then \eqref{LCmass}-\eqref{LCS} yield that $\delta U_+ (\xi,\eta)\defs (\delta p_+,\, \delta \theta_+,\, \delta q_+,\, \delta S_+)^{\top}(\xi,\eta)$ satisfies the following equations:
\begin{align}
  &\partial_{\xi} \big( \frac{1-\overline{M}_+^2(\eta)}{\bar{\rho}_+^2(\eta) \bar{q}_+^3(\eta)} A_+(\eta) \delta{p}_+\big) - \partial_{\eta}\big( A_+(\eta) \delta{\theta}_+\big)= F_1,\label{Rrefeq1}\\
 &\partial_{\xi}\big(\frac{\bar{q}_+(\eta)}{ A_+(\eta) } \delta{\theta}_+\big) + \partial_{\eta} \big(\frac{1}{A_+(\eta)}\delta{p}_+\big) + \frac{g^2}{ \bar{q}_+^3(\eta)A_+(\eta)}\int_{\mathcal{K}}^{\xi} \delta{\theta}_+ (\tau, \eta)\dif \tau  = F_2,\label{Rrefeq2}\\
&\partial_\xi \big( \frac{1}{\bar{\rho}_+(\eta)}\delta p_+ +  \bar{q}_+ (\eta)\delta{q}_+ + \bar{T}_+ (\eta)\delta S_+\big) + g \delta \theta_+ = F_3,\label{Rrefeq3}\\
&\partial_\xi \delta {S}_+ = 0,\label{Rrefeq4}
\end{align}
where $A_+(\eta)\defs e^{-g\int_0^\eta
\frac{1}{\bar{\rho}_+(s)\bar{q}_+^3(s)}\dif s}$ and
\begin{align}
 F_1 \defs &
\partial_{\xi} \big(A_+(\eta)f_{11} + \frac{A_+(\eta)}{\bar{\rho}_+(\eta)\bar{q}_+^3(\eta)}f_{31}\big)+ \partial_{\eta} (A_+(\eta)f_{12})\notag\\
& - A_+'(\eta)f_{12} +A_+(\eta) f_{13}+ \frac{A_+(\eta)}{\bar{\rho}_+(\eta)\bar{q}_+^3(\eta)}f_{32},\label{defsF1}\\
F_2 \defs &
\partial_{\xi} \big(\frac{1}{A_+(\eta)}f_{21}\big) + \partial_{\eta} \big(\frac{1}{A_+(\eta)}f_{22} \big) + \frac{A_+'(\eta)}{A_+^2(\eta)} f_{22} + \frac{1}{A_+(\eta)}f_{23}\notag\\
& +\frac{g}{\bar{\rho}_+(\eta)\bar{q}_+^3(\eta)A_+(\eta)}\big(\delta p_+(\mathcal{K}, \eta) + \bar{\rho}_+(\eta)\bar{q}_+(\eta)\delta{q}_+(\mathcal{K}, \eta) \big)+ \frac{g}{\bar{q}_+^3(\eta)A_+(\eta)} f_{31}(\xi,\eta)\notag\\
&  -  \frac{g}{\bar{q}_+^3(\eta)A_+(\eta)} f_{31}(\mathcal{K},\eta) + \frac{g}{\bar{q}_+^3(\eta)A_+(\eta)}\int_{\mathcal{K}}^{\xi} f_{32} (s,\eta)\dif s,\\
F_3\defs & \partial_{\xi} f_{31} +  f_{32},\\
f_{11}\defs & - \frac{1}{\rho_+ q_+\cos\theta_+}+ \frac{1}{\bar{\rho}_+(\eta)\bar{q}_+(\eta)} - \frac{1}{\bar{\rho}_+^2(\eta)\bar{q}_+(\eta)\bar{c}_+^2(\eta)}
   \delta{p}_+ - \frac{1}{\bar{\rho}_+(\eta)\bar{q}_+^2(\eta)}\delta{q}_+\notag\\
   &- \frac{\psi(\overline{m})-\mathcal{K} - \int_{\eta}^{\overline{m}} \psi'(s)\dif s}{L -\psi(\eta)}\frac{1}{\rho_+ q_+ \cos\theta_+} +\frac{m}{\overline{m}} \frac{(\xi -L)\psi'}{L-\psi}\tan\theta_+,\label{Lf11=}\\
   f_{12}\defs& (\frac{m}{\overline{m}} -1) \delta \theta_+ + \frac{m}{\overline{m}} (\tan\theta_+ - \delta \theta_+), \notag\\
   f_{13}\defs & - \frac{m}{\overline{m}} \frac{\psi'}{L-\psi}\tan\theta_+,\\
   f_{21}\defs &  - q_+\sin\theta_+ + \bar{q}_+(\eta)\delta\theta_+ -\frac{\psi(\overline{m})- \mathcal{K}- \int_{\eta}^{\overline{m}} \psi'(s)\dif s}{L -\psi(\eta)}q_+ \sin\theta_+ \notag\\
    &- \frac{m}{\overline{m}} \frac{(\xi -L)\psi'}{L-\psi}\delta p_+,\\
  f_{22}\defs & - \frac{m-\overline{m}}{\overline{m}}\delta p_+,\\
 f_{23}\defs &\frac{m-\overline{m}}{\overline{m}}\frac{g}{\bar{q}_+(\eta)}
  -g\big(\frac{1}{q_+\cos\theta_+}-\frac{1}{\bar{q}_+(\eta)}+ \frac{1}{\bar{q}_+^2(\eta)} \delta q_+ \big)  + \frac{m}{\overline{m}} \frac{\psi'}{L-\psi}\delta p_+,\\
  f_{31}\defs & \frac{1}{\bar{\rho}_+(\eta)}\delta p_+ +  \bar{q}_+(\eta)\delta{q}_+ + \bar{T}_+(\eta) \delta S_+ - \big(\frac12 q_+^2 + i_+ - (\frac12 \bar{q}_+^2(\eta) + \bar{i}_+(\eta)) \big)\notag\\
  & - \frac{\psi(\overline{m}) - \mathcal{K} -  \int_{\eta}^{\overline{m}} \psi'(s)\dif s}{L -\psi(\eta)} (\frac12 q_+^2 + i_+),\\
  f_{32}\defs & g(\delta \theta_+ - \tan\theta_+ ).\label{Lf32=}
\end{align}

The R.-H. conditions \eqref{RH1}-\eqref{RH4} are reformulated as
\begin{align}
&A_s (\delta{p}_+, \delta{q}_+,\delta{S}_+)^{\top} = (\mathcal{H}_1^{\sharp},\, \mathcal{H}_2^{\sharp},\, \mathcal{H}_3^{\sharp})^\top \defs \overrightarrow{\mathcal{H}},\label{reformulationRH123}\\
&\alpha_4^+ \cdot \delta U_+ - \frac{m}{\overline{m}} [\bar{p}]\delta \psi' = \mathcal{H}_4^{\sharp} (\delta U_+,\delta U_-; \delta \psi'),\label{reformulationRH4}
\end{align}
where \[ A_s\defs  \displaystyle\frac{[\bar{p}]}{\bar{\rho}_+\bar{q}_+}\begin{pmatrix}
	-\displaystyle\frac{1}{\bar{\rho}_+ \bar{c}_+^2}&-\displaystyle\frac{1}{\bar{q}_+} & \displaystyle\frac{1}{\gamma c_v}\\
1 - \displaystyle\frac{\bar{p}_+}{\bar{\rho}_+ \bar{c}_+^2}& \bar{\rho}_+\bar{q}_+ - \displaystyle\frac{\bar{p}_+}{\bar{q}_+}& \displaystyle\frac{\bar{p}_+}{\gamma c_v}\\
	\displaystyle\frac{\bar{q}_+}{[\bar{p}]}& \displaystyle\frac{\bar{\rho}_+\bar{q}_+^2}{[\bar{p}]}& \displaystyle\frac{\bar{p}_+}{(\gamma - 1)c_v}\displaystyle\frac{\bar{q}_+}{[\bar{p}]}
	\end{pmatrix},\]
and
\begin{align}
&\mathcal{H}_j^{\sharp}(\delta U_+,\delta U_-): = \alpha_j^+\cdot \delta U_+ - \mathcal{H}_j(U_+, U_-),\quad {j = 1,2,3},\label{2.73}\\
&\mathcal{H}_4^{\sharp} (\delta U_+,\delta U_-; \delta\psi'): = \alpha_4^+ \cdot \delta U_+ - \frac{m}{\overline{m}} [\bar{p}]\delta \psi' - \mathcal{H}_4(U_+, U_-;\delta \psi')\label{xxl},\\
&\mathcal{H}_1(U_+, U_-) : = \Big[\displaystyle\frac{1}{\rho u}\Big][p] + \Big[\displaystyle\frac{v}{u}\Big][v]= 0,\label{LG1}\\
&\mathcal{H}_2(U_+, U_-) : = \Big[u+ \displaystyle\frac{p}{\rho u}\Big] [p]  + \Big[\displaystyle\frac{p v}{u}\Big][v]=0,\label{LG2}\\
&\mathcal{H}_3(U_+, U_-): = \Big[\frac12 q^2+i\Big]=0,\label{LG3}\\
&\mathcal{H}_4(U_+, U_-;\delta \psi'): = [v] - \frac{m}{\overline{m}}\delta \psi'[p]=0,\\
&{\mathbf{\alpha}}_4^{+} = {{\nabla_{{U}_+}}}\mathcal{H}_4(\overline{U}_+, \overline{U}_-;0)= (\bar{q}_+,\, 0,\,0 ,\, 0)^{\top}.\label{alpha4+}
\end{align}
Direct calculations yield that
\begin{align}\label{AS}
  \det A_s = \frac{1}{(\gamma -1)c_v} \frac{[\bar{p}]^2 \bar{p}_+}{(\bar{\rho}_+\bar{q}_+)^3}(1 - \overline{M}_+^2)\neq 0,\quad \text{as}\quad \overline{M}_+\neq 1.
\end{align}
Thus, \eqref{reformulationRH123} and \eqref{reformulationRH4} imply that
\begin{align}
&(\delta{p}_+,\,\delta{q}_+,\, \delta{S}_+) \defs (\frac{\bar{\rho}_+\bar{q}_+^2 }{1 - \overline{M}_+^2}{h}_1,\, \frac{1}{\bar{\rho}_+ \bar{q}_+}{h}_2,\,{h}_3) =A_s^{-1} \overrightarrow{\mathcal{H}}, \label{As-1H}\\
&\delta \psi' = {h}_4\defs \frac{\overline{m}}{m} \frac{\alpha_4^+ \delta U_+ - \mathcal{H}_4^{\sharp} (\delta U_+,\delta U_-; \delta \psi')}{ [\bar{p}]}.\label{h4==}
\end{align}

Then in the subsonic domain, we consider the following problem.
\begin{defn}[Subsonic solution behind the shock front]\label{RSP}
We say $\pr{ \delta{U}_+(\xi,\eta);\ \psi(\eta)}$
is a subsonic solution behind the shock front to the following boundary value problem:
\begin{enumerate}
\item $\delta U_+$ satisfies the Euler system \eqref{Rrefeq1}-\eqref{Rrefeq4} in
 $\Omega_+^{\mathcal{K}}$;

\item On the exit of the nozzle, the pressure is prescribed by:
\begin{align}\label{expre}
\delta p_+ (L, \eta) = \sigma p_{ex}(y(L,\eta;\overline{U}_+ +\delta U_+)), \quad  \text{on} \quad \Gamma_{ex},
\end{align}
where
\vspace{-0.3cm}
\begin{align}
y(L,\eta;\overline{U}_+ +\delta U_+) = \frac{m}{\overline{m}}\int_{0}^{\eta}\displaystyle\frac{1}{\rho_+ q_+\cos\theta_+ (L,s)} \dif s;
\end{align}

\item On the nozzle walls, the slip boundary conditions are as below:
\begin{align}
&\delta \theta_+(\xi) = 0,  &\text{on}&\quad \Gamma_0 \cap \overline{\Omega_+^{\mathcal{K}}},
\label{eq:theta2+re}\\
&\delta \theta_+(\xi) = \sigma\Theta\big( \frac{L(\psi(\overline{m}) - \mathcal{K} ) + (L-\psi(\overline{m}))\xi }{L-\mathcal{K}} \big),  &\text{on} &\quad \Gamma_{\overline{m}} \cap \overline{\Omega_+^{\mathcal{K}}};
\label{eq:theta4+re}
\end{align}

\item On the fixed shock front $\Gamma_s^{\mathcal{K}}$, \eqref{As-1H} and \eqref{h4==} hold.
\end{enumerate}
\end{defn}

\section{Linearized problem in the subsonic domain}
In order to solve the nonlinear free boundary value problem in the sense of Definition \ref{RSP}, we first consider a linearized problem of equations \eqref{Rrefeq1}-\eqref{Rrefeq4} with boundary conditions \eqref{As-1H}-\eqref{eq:theta4+re} and establish a well-posedness theorem for it in this section.
Let
\begin{align}
&Q_1\defs (\mathcal{K},0),\,\, Q_2\defs (L,0),\,\, Q_3\defs (L,\overline{m}),\,\, Q_4\defs (\mathcal{K},\overline{m}),\label{Qidefs=}\\
&\Gamma_0^{\mathcal{K}} \defs \Gamma_0 \cap \overline{\Omega_+^{\mathcal{K}}}, \,\, \Gamma_{\overline{m}}^{\mathcal{K}} \defs \Gamma_{\overline{m}} \cap \overline{\Omega_+^{\mathcal{K}}}.\label{Gammadefs=}
\end{align}
Assume that
\begin{align}
&f_{ij} \in C_{1,\alpha}^{(-\alpha; \{ \overline{\Gamma_0^{\mathcal{K}}}\cup \overline{\Gamma_{\overline{m}}^{\mathcal{K}}}\})}(\Omega_+^{\mathcal{K}}), \quad i=1,2,3,\, j=1,2,3, \,\, \text{and}\,\, f_{33} =0,\label{Assumpf}\\
& h_j\in C_{1,\alpha}^{(-\alpha; Q_1\cup Q_4)}(\Gamma_s^{\mathcal{K}}), \, j=1,2,3,4,\label{Assumph}\\
&p_{ex}(y(L,\eta;\overline{U}_+ +\delta U_+))\in C_{1,\alpha}^{(-\alpha; Q_2\cup Q_3)}(\Gamma_{ex}), \quad \Theta(\xi)\in C^{2,\alpha}(\Gamma_{\overline{m}}^{\mathcal{K}}).\label{assumebry}
\end{align}

The well-posedness of a linearized problem under assumptions \eqref{Assumpf}-\eqref{assumebry} will be established by proving the following theorem.
\begin{thm}\label{Rthm}
Assume that \eqref{assumpthm} and \eqref{Assumpf}-\eqref{assumebry} hold, then there exists a unique solution $(\delta p_+,\, \delta \theta_+,\, \delta q_+,\, \delta S_+; \delta \psi')$ to equations \eqref{Rrefeq1}-\eqref{Rrefeq4} with boundary conditions \eqref{As-1H}-\eqref{eq:theta4+re} if and only if
  \begin{align}\label{solvabilityhat}
 &\int_0^{\overline{m}}\frac{A_+(\eta)}{\bar{\rho}_+(\eta) \bar{q}_+(\eta)} {h}_1(\eta)\dif \eta + A_+(\eta) \int_{\mathcal{K}}^L\sigma \Theta \big( \frac{L(\psi(\overline{m}) - \mathcal{K}) + (L-\psi(\overline{m}))\xi }{L-\mathcal{K}} \big)\dif \xi\notag\\
 =& \int_0^{\overline{m}}\frac{1-\overline{M}_+^2(\eta)}{\bar{\rho}_+^2(\eta) \bar{q}_+^3(\eta)} A_+(\eta) \sigma p_{ex}
   \big(\frac{m}{\overline{m}}\int_{0}^{\eta}\displaystyle\frac{1}{\rho_+ q_+\cos\theta_+ (L,t)} \dif t\big) \dif\eta\notag\\
    &-\int_{0}^{\overline{m}}\int_{\mathcal{K}}^{L} F_1\dif \xi \dif \eta.
\end{align}
Moreover, the solution, if it exists, satisfies
  \begin{align}\label{pthetaqses}
&\|(\delta p_+,\delta \theta_+) \|_{1,\alpha;\Omega_+^{\mathcal{K}}}^{(-\alpha;\{Q_i\})} +  \|(\delta q_+, \delta S_+)\|_{1,\alpha;\Omega_+^{\mathcal{K}}}^{(-\alpha;\{\overline{\Gamma_0^{\mathcal{K}}}\cup \overline{\Gamma_{\overline{m}}^{\mathcal{K}}}\})} +\| {\psi}'\|_{1,\alpha;\Gamma_s^{\mathcal{K}}}^{(-\alpha;\{Q_1,Q_4\})}\notag\\
\leq & C\Big(\sum_{k=1}^3 \|F_k \|_{0,\alpha;\Omega_+^{\mathcal{K}}}^{(1-\alpha;\{\overline{\Gamma_0^{\mathcal{K}}}\cup \overline{\Gamma_{\overline{m}}^{\mathcal{K}}}\})}  + \sigma \|p_{ex}\|_{1,\alpha;\Gamma_{ex}}^{(-\alpha;\{Q_2, Q_3\})} \notag\\
&\qquad + \sigma \|\Theta\|_{C^{2,\alpha}(\Gamma_{\overline{m}}^{\mathcal{K}})} +\sum_{j=1}^4 \|h_j\|_{1,\alpha;\Gamma_s^{\mathcal{K}}}^{(-\alpha;\{Q_1, Q_4\})} \Big),
\end{align}
where the constant $C$ depends on $\overline{U}_{\pm}$, $L$ and $\alpha$.

\end{thm}

\begin{proof}
This proof is divided into three steps.

\emph{Step 1:} By \eqref{Rrefeq3}-\eqref{Rrefeq4} and \eqref{As-1H}, one has
\begin{align}
  \delta {S}_+(\xi,\eta) =& \delta{S}_+(\mathcal{K}, \eta) = {h}_3(\eta) ,\label{S+=eq}\\
  \delta{q}_+(\xi,\eta)=&\frac{\bar{q}_+(\eta)}{1-\overline{M}_+^2(\eta)}{h}_1(\eta)+ \frac{1}{\bar{\rho}_+(\eta) \bar{q}_+(\eta)}{h}_2(\eta) -
  \frac{1}{\bar{\rho}_+(\eta)\bar{q}_+(\eta)}\delta{p}_+(\xi,\eta)\notag\\
  & - \frac{g}{\bar{q}_+(\eta)} \int_{\mathcal{K}}^{\xi} \delta{\theta}_+(\tau,\eta)\dif \tau + \frac{1}{\bar{q}_+(\eta)}\int_{\mathcal{K}}^{\xi}F_3(\delta U_+, \delta\psi', \psi(\overline{m}); \mathcal{K}) \dif \tau.\label{q+=eq}
\end{align}
Thus, it suffices to solve $(\delta p_+, \delta \theta_+)$ to equations \eqref{Rrefeq1}-\eqref{Rrefeq2} with boundary conditions \eqref{expre}-\eqref{eq:theta4+re} and \eqref{As-1H}.
An auxiliary problem is introduced for $(\delta{p}_+^{\sharp}, \delta{\theta}_+^{\sharp})$ as  
\begin{align}
  &\partial_{\xi} \big( \frac{1-\overline{M}_+^2(\eta)}{\bar{\rho}_+^2(\eta) \bar{q}_+^3(\eta)} A_+(\eta) \delta{p}_+^{\sharp}\big) - \partial_{\eta}\big( A_+(\eta) \delta{\theta}_+^{\sharp}  \big)= F_1,\label{sharpeq1}\\
& \partial_{\xi}\big(\frac{\bar{q}_+(\eta)}{ A_+(\eta) } \delta{\theta}_+^{\sharp} \big) + \partial_{\eta} \big(\frac{1}{A_+(\eta)}\delta{p}_+^{\sharp}\big)=0,\label{sharpeq2}
\end{align}
with the boundary conditions:
\begin{align}
   \delta{p}_+^{\sharp}(\mathcal{K}, \eta) =& \frac{\bar{\rho}_+(\eta)\bar{q}_+^2(\eta)}{1 - \overline{M}_+^2(\eta)}\sigma  H ,\quad \delta{p}_+^{\sharp}(L,\eta) = 0,\label{deltaptheta1}\\
  \delta{\theta}_+^{\sharp}(\xi,0) =& 0,\quad \delta{\theta}_+^{\sharp}(\xi,\overline{m}) = 0,\label{deltaptheta2}
\end{align}
where $H$ is a constant and satisfies
\begin{align}\label{NeumanSolvability}
  \int_{0}^{\overline{m}}\int_{\mathcal{K}}^{L} F_1\dif \xi \dif \eta =-\sigma H \int_0^{\overline{m}}\frac{ A_+(\eta) }{\bar{\rho}_+(\eta) \bar{q}_+(\eta)} \dif \eta.
\end{align}
Let $(\widehat{\delta{p}_+},\, \widehat{\delta{\theta}_+}) \defs (\delta{p}_+ - \delta{p}_+^{\sharp}, \, \delta{\theta}_+ - \delta{\theta}_+^{\sharp})$. By applying \eqref{Rrefeq1}-\eqref{Rrefeq2}, \eqref{expre}-\eqref{eq:theta4+re}, \eqref{As-1H} and (\eqref{sharpeq1}-\eqref{deltaptheta2}), we find that $(\widehat{\delta{p}_+},\, \widehat{\delta{\theta}_+})$ satisfies the following equations:
\begin{align}
  &\partial_{\xi} \big( \frac{1-\overline{M}_+^2(\eta)}{\bar{\rho}_+^2(\eta) \bar{q}_+^3(\eta)} A_+(\eta) \widehat{\delta{p}_+}\big) - \partial_{\eta}\big( A_+(\eta) \widehat{\delta{\theta}_+} \big)= 0,\label{NLdoteq1}\\
 &\partial_{\xi}\big(\frac{\bar{q}_+(\eta)}{ A_+(\eta) } \widehat{\delta{\theta}_+} \big) + \partial_{\eta} \big(\frac{1}{A_+(\eta)}\widehat{\delta{p}_+}\big) + \frac{g^2}{ \bar{q}_+^3(\eta)A_+(\eta)}\int_{\mathcal{K}}^{\xi} \widehat{\delta{\theta}_+}(s, \eta)\dif s \notag\\
 =&F_2 -  \frac{g^2}{ \bar{q}_+^3(\eta)A_+(\eta)}\int_{\mathcal{K}}^{\xi} \delta{\theta}_+^{\sharp}(s, \eta)\dif s,\label{NLdoteq2}
\end{align}
with the boundary conditions:
\begin{align}
&\widehat{\delta p_+} (\mathcal{K}, \eta) = \frac{\bar{\rho}_+(\eta)\bar{q}_+^2(\eta)}{1 - \overline{M}_+^2(\eta)}\big(h_1(\eta) -\sigma  H\big),\quad  \text{on} \quad \Gamma_s^{\mathcal{K}},\label{NRHg1+}\\
&\widehat{\delta p_+} (L, \eta) = \sigma p_{ex}\big(\frac{m}{\overline{m}}\int_{0}^{\eta}\displaystyle\frac{1}{\rho_+ q_+\cos\theta_+ (L, \tau)} \dif \tau\big), \quad  \text{on} \quad \Gamma_{ex},\\
&\widehat{\delta \theta_+} (\xi,0) = 0,\quad \text{on}\quad \Gamma_0^{\mathcal{K}},\label{Neq:theta2+re}\\
&\widehat{\delta \theta_+} (\xi,\overline{m}) = \sigma\Theta\big( \frac{L(\psi(\overline{m}) - \mathcal{K}) + (L-\psi(\overline{m}))\xi }{L-\mathcal{K}} \big),\quad \text{on} \quad \Gamma_{\overline{m}}^{\mathcal{K}}.
\label{Neq:theta4+re}
\end{align}
Then we will establish the well-posedness of problems \eqref{sharpeq1}-\eqref{deltaptheta2} and \eqref{NLdoteq1}-\eqref{Neq:theta4+re}, respectively.

\emph{Step 2:} In this step, we will establish the well-posedness of problem \eqref{sharpeq1}-\eqref{deltaptheta2}.

Notice that \eqref{sharpeq2} yields that there exists a function ${\Phi}$ such that
\begin{align}\label{deltPhi}
  \nabla {\Phi} = (\partial_{\xi} {\Phi}, \, \partial_{\eta} {\Phi}) = \big( \frac{1}{A_+(\eta)}\delta{p}_+^{\sharp},\, - \frac{\bar{q}_+(\eta)}{ A_+(\eta) } \delta{\theta}_+^{\sharp}\big).
\end{align}
Substituting \eqref{deltPhi} into \eqref{sharpeq1}, one has
\begin{align}\label{ellipticPhi}
  \partial_{\xi} \big( \frac{1-\overline{M}_+^2(\eta)}{\bar{\rho}_+^2(\eta) \bar{q}_+^3(\eta)} A_+^2(\eta)\partial_{\xi} {\Phi}\big) + \partial_{\eta}\big( \frac{A_+^2(\eta)}{\bar{q}_+(\eta)} \partial_{\eta} {\Phi}\big)= F_1,\quad \text{in}\quad \Omega_+^{\mathcal{K}}.
\end{align}
The boundary conditions \eqref{deltaptheta1}-\eqref{deltaptheta2} become
\begin{align}
  \partial_{\xi} \Phi(\mathcal{K}, \eta) =&\frac{1}{A_+(\eta)}\frac{\bar{\rho}_+(\eta)\bar{q}_+^2(\eta)}{1 - \overline{M}_+^2(\eta)} \sigma H ,\quad \partial_{\xi} \Phi(L,\eta) = 0,\label{Phibry1}\\
  \partial_{\eta}\Phi(\xi,0) =& 0,\quad \partial_{\eta} \Phi(\xi,\overline{m}) = 0.\label{Phibry2}
\end{align}

First, we prove the unique existence of the weak solution $\Phi$ to the problem \eqref{ellipticPhi}-\eqref{Phibry2}.
Notice that $\Phi$ is a weak solution to the problem \eqref{ellipticPhi}-\eqref{Phibry2} if and only if for any test function $\zeta \in H^1 (\Omega_+^{\mathcal{K}})$,
\begin{align}\label{zetaweak}
&-\sigma H \int_{0}^{\overline{m}} \frac{ A_+(\eta)}{\bar{\rho}_+(\eta) \bar{q}_+(\eta)} \zeta(\mathcal{K}, \eta)\dif \eta\notag\\
&-\int_{\Omega_+^{\mathcal{K}}}\big(\frac{1-\overline{M}_+^2(\eta)}{\bar{\rho}_+^2(\eta) \bar{q}_+^3(\eta)} A_+^2(\eta)\partial_{\xi} {\Phi} \partial_{\xi}\zeta + \frac{A_+^2(\eta)}{\bar{q}_+} \partial_{\eta} {\Phi} \partial_{\eta}\zeta \big)\dif\xi \dif \eta \notag\\
= &\int_{\Omega_+^{\mathcal{K}}} F_1 \zeta(\xi,\eta)\dif\xi \dif \eta.
\end{align}
Choosing $\zeta=1$ in \eqref{zetaweak}, it follows that \eqref{NeumanSolvability}.

Next, we will prove that if \eqref{NeumanSolvability} holds, there exists a unique weak solution $\Phi$ to problem \eqref{ellipticPhi}-\eqref{Phibry2}.
Applying trace theorem (see \cite[Theorem 1.5.2.1]{PG}), there exists a function $\mathcal{G}(\xi,\eta)\in H^2( \Omega_+^{\mathcal{K}})$ such that
\begin{align}
  \widehat{\Phi}\defs \Phi - \mathcal{G}
\end{align}
satisfies the following boundary value problem:
\begin{align}\label{hateq=}
&\partial_{\xi} \big( \frac{1-\overline{M}_+^2(\eta)}{\bar{\rho}_+^2(\eta) \bar{q}_+^3(\eta)} A_+^2(\eta)\partial_{\xi} \widehat{\Phi}\big) + \partial_{\eta}\big( \frac{A_+^2(\eta)}{\bar{q}_+(\eta)} \partial_{\eta} \widehat{\Phi}\big)\notag\\
=& F_1 - \partial_{\xi} \big( \frac{1-\overline{M}_+^2(\eta)}{\bar{\rho}_+^2(\eta) \bar{q}_+^3(\eta)} A_+^2(\eta)\partial_{\xi} \mathcal{G}\big) + \partial_{\eta}\big( \frac{A_+^2(\eta)}{\bar{q}_+(\eta)} \partial_{\eta} \mathcal{G}\big) \defs \widehat{F_1},\quad \text{in}\quad \Omega_+^{\mathcal{K}},\\
&\partial_{\xi}\widehat{\Phi}(\mathcal{K}, \eta) =0 ,\, \partial_{\xi} \widehat{\Phi}(L,\eta) = 0,\,
\partial_{\eta}\widehat{\Phi}(\xi,0) = 0,\, \partial_{\eta} \widehat{\Phi}(\xi,\overline{m}) = 0.
\end{align}
\eqref{NeumanSolvability} yields that
\begin{align}\label{intwihatF=0}
 \int_{\Omega_+^{\mathcal{K}}} \widehat{F_1} \dif \xi \dif\eta =0.
\end{align}
Define a subspace of $H^1(\Omega_+^{\mathcal{K}})$ as
$
\mathcal{S} \defs \big\{\Phi_* \in H^1 (\Omega_+^{\mathcal{K}})\big| \int_{\Omega_+^{\mathcal{K}}} \Phi_* \dif \xi \dif \eta =0\big\}.
$
Let
\begin{align*}
\mathcal{B}[\widehat{\Phi}, \zeta] \defs -\int_{\Omega_+^{\mathcal{K}}}
\big(\frac{1-\overline{M}_+^2(\eta)}{\bar{\rho}_+^2(\eta) \bar{q}_+^3(\eta)} A_+^2(\eta)\partial_{\xi} \widehat{\Phi}\partial_{\xi}\zeta + \frac{A_+^2(\eta)}{\bar{q}_+(\eta)} \partial_{\eta} \widehat{\Phi} \partial_{\eta}\zeta \big)\dif\xi \dif \eta.
  \end{align*}
It is easy to see that for any $\widehat{\Phi}\in\mathcal{S}$ and $\zeta\in H^1(\Omega_+^{\mathcal{K}})$
\begin{align}
  | \mathcal{B}[\widehat{\Phi}, \zeta]|\leq& \big( \big\|\frac{1-\overline{M}_+^2(\eta)}{\bar{\rho}_+^2(\eta) \bar{q}_+^3(\eta)} A_+^2(\eta)\big\|_{L^\infty([0,\overline{m}])} +  \big\|\frac{A_+^2(\eta)}{\bar{q}_+} \big\|_{L^\infty([0,\overline{m}])} \big)  \int_{\Omega_+^{\mathcal{K}}} |D \widehat{\Phi}| |D \zeta|\dif \xi\dif \eta\notag\\
  \leq& C\| \widehat{\Phi}\|_{H^1(\Omega_+^{\mathcal{K}})} \|\zeta\|_{H^1(\Omega_+^{\mathcal{K}})},
\end{align}
where the positive constant $C$ only depends on $\overline{U}_+$.
By the Poincar\'{e} inequality, we have for any $\widehat{\Phi}\in\mathcal{S}$
\begin{align}
\|\widehat{\Phi}\|_{L^2(\Omega_+^{\mathcal{K}})}^2 = \int_{\Omega_+^{\mathcal{K}}}|\widehat{\Phi} - \int_{\Omega_+^{\mathcal{K}}} \widehat{\Phi} \dif \xi \dif \eta|^2 \dif\xi\dif\eta\leq C\int_{\Omega_+^{\mathcal{K}}}|D\widehat{\Phi}|^2 \dif \xi\dif\eta.
\end{align}
Thus for any $\widehat{\Phi} \in\mathcal{S}$
\begin{align}
  |\mathcal{B} [\widehat{\Phi},\widehat{\Phi}]| \geq C \|\widehat{\Phi}\|_{H^1(\Omega_+^{\mathcal{K}})}^2.
\end{align}
By applying the Lax-Milgram theorem, for all $\widehat{F_1}\in L^2$ with \eqref{intwihatF=0}, there exists a unique $\Phi_* \in \mathcal{S}$ such that for any $\zeta\in H^1(\Omega_+^{\mathcal{K}})$
\begin{align*}
(\widehat{F_1}, \zeta)_{L^2(\Omega_+^{\mathcal{K}})} = &(\widehat{F_1}, \int_{\Omega_+^{\mathcal{K}}} \zeta \dif \xi\dif\eta)_{L^2(\Omega_+^{\mathcal{K}})} + (\widehat{F_1}, \zeta - \int_{\Omega_+^{\mathcal{K}}} \zeta \dif \xi\dif\eta)_{L^2(\Omega_+^{\mathcal{K}})} \notag\\
= &B[\Phi_*, \zeta - \int_{\Omega_+^{\mathcal{K}}} \zeta \dif \xi\dif\eta] =B[\Phi_*, \zeta].
\end{align*}
Therefore, $\Phi_* \in \mathcal{S} \subset H^1(\Omega_+^{\mathcal{K}})$ is a weak solution. Notice that $(\Phi_* + \text{constant.})$ is also a weak solution and the solution is unique in $\mathcal{S}$. Thus, $\Phi_* \in H^1(\Omega_+^{\mathcal{K}})$ is unique up to a constant.
Moreover, choose the solution $\widehat{\Phi}$ with the additional condition $\int_{\Omega_+^{\mathcal{K}}} \widehat{\Phi} \dif \xi \dif\eta=0$, then the following estimate holds
\begin{align}
\|\widehat{\Phi}\|_{H^1(\Omega_+^{\mathcal{K}})} \leq C \| \widehat{F_1} \|_{L^2(\Omega_+^{\mathcal{K}})}.
\end{align}
Applying \eqref{NeumanSolvability} and \eqref{hateq=}, it follows that
\begin{align}\label{H1es=}
\|{\Phi}\|_{H^1(\Omega_+^{\mathcal{K}})} \leq C \| F_1 \|_{0,\alpha;\Omega_+^{\mathcal{K}}}^{(1-\alpha;\{\overline{\Gamma_0^{\mathcal{K}}}\cup \overline{\Gamma_{\overline{m}}^{\mathcal{K}}}\})}.
\end{align}
By \eqref{H1es=} and \cite[Theorem 5.36]{GM13}, one has
\begin{align}\label{Linftyes=}
  \|{\Phi}\|_{L^\infty(\Omega_+^{\mathcal{K}})} \leq C \| F_1 \|_{0,\alpha;\Omega_+^{\mathcal{K}}}^{(1-\alpha;\{\overline{\Gamma_0^{\mathcal{K}}}\cup \overline{\Gamma_{\overline{m}}^{\mathcal{K}}}\})}.
\end{align}

Now we estimate $\Phi$ near the corner $(\mathcal{K}, 0)$, and the other corners can be treated similarly.
Let $\mathcal{P}$ be a small domain including the point $(\mathcal{K}, 0)$, whose boundary is
smooth except the point $(\mathcal{K}, 0)$. By applying $\partial_{\eta}\Phi(\xi,0) = 0$ in \eqref{Phibry2}, one can extend the solution across the nozzle wall evenly by defining
\begin{align*}
\widetilde{\Phi}(\xi,\eta) \defs \begin{cases}
\Phi(\xi,\eta),\quad 0<\eta<\overline{m},\\
 \Phi (\xi, -\eta),\quad \eta\leq 0.
\end{cases}
\end{align*}
By \eqref{defsF1}, we also define
\begin{align*}
  \widetilde{f_{kj}}(\xi,\eta) \defs \begin{cases}
f_{kj}(\xi,\eta),\quad 0<\eta<\overline{m},\\
 f_{kj} (\xi, -\eta),\quad \eta\leq 0,
\end{cases}
\widetilde{H^{\sharp}}(\mathcal{K}, \eta) \defs \begin{cases}
H^{\sharp}(\mathcal{K}, \eta),\quad 0<\eta<\overline{m},\\
H^{\sharp}(\mathcal{K}, -\eta),\quad \eta \leq 0,
\end{cases}
\end{align*}
where $k=1, j=1,2,3$ and $k=3, j=1,2$, $H^{\sharp}(\mathcal{K}, \eta)\defs \frac{1}{A_+(\eta)}\frac{\bar{\rho}_+(\eta)\bar{q}_+^2(\eta)}{1 - \overline{M}_+^2(\eta)} \sigma H$.
Let $\widetilde{\mathcal{P}} \defs \mathcal{P} \cup \{(\xi, -\eta): (\xi,\eta) \in \mathcal{P}  \}$. Then $\widetilde{\Phi}$ is the solution to the following boundary value problem
\begin{align}
&\partial_{\xi} \big( \frac{1-\overline{M}_+^2(\eta)}{\bar{\rho}_+^2(\eta) \bar{q}_+^3(\eta)} A_+^2(\eta)\partial_{\xi}\widetilde{\Phi}\big) + \partial_{\eta}\big( \frac{A_+^2(\eta)}{\bar{q}_+(\eta)} \partial_{\eta} \widetilde{\Phi}\big)= \widetilde{F_1},\quad \text{in}\quad \widetilde{\mathcal{P}},\label{tildePhi=}\\
&\partial_{\xi} \widetilde{\Phi}(\mathcal{K}, \eta) =\widetilde{H^{\sharp}}(\mathcal{K}, \eta), \quad \text{on}\quad \partial\widetilde{\mathcal{P}}\cap \{\xi=\mathcal{K}\},\\
& \widetilde{\Phi} =  \widetilde{\Phi}, \quad \text{on}\quad \partial\widetilde{\mathcal{P}}\backslash \{\xi=\mathcal{K}\}.
\end{align}
By applying the Schauder interior and boundary estimates in \cite[Theorem 8.32, Corollary 8.36]{GT} and \eqref{Linftyes=},
for any smooth domain $\widetilde{\mathcal{P}}^{\sharp} \subset\widetilde{\mathcal{P}}$ and \eqref{defsF1}, \eqref{Assumpf}, one has
\begin{align}\label{pPes}
\|\widetilde{\Phi}\|_{C^{1,\alpha}(\widetilde{\mathcal{P}}^{\sharp})} \leq& C \big(\sum_{j=1}^3 \| \widetilde{f_{1j}}\|_{C^{0,\alpha}(\widetilde{\mathcal{P}})} + \sum_{j=1}^2\|\widetilde{f_{3j}}\|_{C^{0,\alpha}(\widetilde{\mathcal{P}})} + \|\widetilde{H^\sharp}\|_{C^{0,\alpha}(\partial\widetilde{\mathcal{P}}\cap \{\xi=\mathcal{K}\})} + \|\widetilde{\Phi}\|_{C^{0,\alpha}(\partial \widetilde{\mathcal{P}} \backslash \{\xi = \mathcal{K}\})} \big)\notag\\
\leq & C \big(\| \widetilde{F_1} \|_{0,\alpha;\widetilde{\mathcal{P}}}^{(1-\alpha;\partial\widetilde{\mathcal{P}}\cap \{\xi=\mathcal{K}\})} + \|\widetilde{H^\sharp}\|_{C^{0,\alpha}(\partial\widetilde{\mathcal{P}}\cap \{\xi=\mathcal{K}\})} + \|\widetilde{\Phi}\|_{C^{0,\alpha}(\partial \widetilde{\mathcal{P}} \backslash \{\xi = \mathcal{K}\})} \big).
\end{align}
It yields that
\begin{align}
\|\Phi\|_{C^{1,\alpha}(\widetilde{\mathcal{P}}^{\sharp}\cap \overline{\mathcal{P}})} \leq \|\widetilde{\Phi}\|_{C^{1,\alpha}(\widetilde{\mathcal{P}}^{\sharp})} \leq C \| F_1 \|_{0,\alpha;\Omega_+^{\mathcal{K}}}^{(1-\alpha;\{\overline{\Gamma_0^{\mathcal{K}}}\cup \overline{\Gamma_{\overline{m}}^{\mathcal{K}}}\})}.
\end{align}
Then by Schauder
estimates in \cite[Chapter 8]{GT} and the standard scaling argument, one has
\begin{align}\label{Phies==}
\|\Phi\|_{2,\alpha;\Omega_+^{\mathcal{K}}}^{(-1-\alpha;\partial\Omega_+^{\mathcal{K}})} \leq C \| F_1 \|_{0,\alpha;\Omega_+^{\mathcal{K}}}^{(1-\alpha;\{\overline{\Gamma_0^{\mathcal{K}}}\cup \overline{\Gamma_{\overline{m}}^{\mathcal{K}}}\})}.
\end{align}
Applying \eqref{deltPhi} and \eqref{Phies==}, it yields that \eqref{pthetasharp}.
Thus, there exists a unique solution $(\delta{p}_+^{\sharp}, \delta{\theta}_+^{\sharp})$ to the problem \eqref{sharpeq1}-\eqref{deltaptheta2} if and only if \eqref{NeumanSolvability} holds.
In addition,
\begin{align}\label{pthetasharp}
\|\delta{p}_+^{\sharp}\|_{1,\alpha;\Omega_+^{\mathcal{K}}}^{(-\alpha;\{Q_i\})} + \|\delta{\theta}_+^{\sharp}\|_{1,\alpha;\Omega_+^{\mathcal{K}}}^{(-\alpha;\{Q_i\})} \leq C\| F_1 \|_{0,\alpha;\Omega_+^{\mathcal{K}}}^{(1-\alpha;\{\overline{\Gamma_0^{\mathcal{K}}}\cup \overline{\Gamma_{\overline{m}}^{\mathcal{K}}}\})},
\end{align}
where the constant $C$ depends on $\overline{U}_{\pm}$, $L$ and $\alpha$.

\emph{Step 3:} In this step, we will establish the well-posedness of the problem \eqref{NLdoteq1}-\eqref{Neq:theta4+re}.

Notice that \eqref{NLdoteq1} yields that there exists a function $\Psi$ such that
\begin{align}\label{NPsi}
  \nabla \Psi = (\partial_{\xi} \Psi,\, \partial_{\eta}\Psi) = \big(  A_+(\eta) \widehat{\delta{\theta}_+},\,  \frac{1-\overline{M}_+^2(\eta)}{\bar{\rho}_+^2(\eta) \bar{q}_+^3(\eta)} A_+(\eta) \widehat{\delta{p}_+}\big).
\end{align}
Substituting \eqref{NPsi} into \eqref{NLdoteq2}, one has
\begin{align}
 & \partial_{\xi}\big(\frac{\bar{q}_+(\eta)}{ A_+^2(\eta) } \partial_{\xi} \Psi \big) + \partial_{\eta} \big(\frac{1}{A_+^2(\eta)} \frac{\bar{\rho}_+^2(\eta) \bar{q}_+^3(\eta)}{1-\overline{M}_+^2(\eta)}\partial_{\eta} \Psi \big) + \frac{g^2}{ \bar{q}_+^3(\eta)A_+^2(\eta)}
 \Psi \notag\\
=&F_2 - \frac{g^2}{ \bar{q}_+^3(\eta)A_+(\eta)}\int_{\mathcal{K}}^{\xi} \delta{\theta}_+^{\sharp} (\tau, \eta)\dif \tau + \frac{g^2}{ \bar{q}_+^3(\eta)A_+^2(\eta)} \Psi(\mathcal{K},\eta)\defs \mathcal{F}_2.\label{eq2s}
\end{align}
Without loss of the generality, we assume that ${\Psi}(\mathcal{K}, 0)=0$. Then \eqref{NRHg1+}-\eqref{Neq:theta4+re} yield that
\begin{align}
\Psi(\xi,0) =& 0,\\
   \Psi(\mathcal{K}, \eta) =& \int_0^{\eta}\frac{A_+(t)}{\bar{\rho}_+(t) \bar{q}_+(t)}  \big(h_1(\mathcal{K}, t) - \sigma H\big)\dif t,\\
   \Psi(L, \eta) =& \int_0^{\eta}\frac{1-\overline{M}_+^2(t)}{\bar{\rho}_+^2(t) \bar{q}_+^3(t)} A_+(t) \sigma p_{ex}
   \big(\frac{m}{\overline{m}}\int_{0}^{t}\displaystyle\frac{1}{\rho_+ q_+\cos\theta_+ (L,s)} \dif s\big) \dif t,\\
   \Psi (\xi,\overline{m})
   =&\int_0^{\overline{m}}\frac{1-\overline{M}_+^2(\eta)}{\bar{\rho}_+^2(\eta) \bar{q}_+^3(\eta)} A_+(\eta) \sigma p_{ex}
   \big(\frac{m}{\overline{m}}\int_{0}^{\eta}\displaystyle\frac{1}{\rho_+ q_+\cos\theta_+ (L, s)} \dif s\big) \dif\eta\notag\\
   & -A_+(\eta) \int_{\xi}^L\sigma \Theta\big( \frac{L(\psi(\overline{m}) - \mathcal{K}) + (L-\psi(\overline{m}))t }{L-\mathcal{K}} \big)\dif t.\label{brym=}
\end{align}
By $(\widehat{\delta p_+},\, \widehat{\delta \theta_+})\in C_{1,\alpha}^{(-\alpha;\{Q_i\})}(\Omega_+^{\mathcal{K}})^2 $ and \eqref{NPsi}, it is easy to obtain the following continuous condition
\begin{align}\label{continuousK}
   \Psi (\xi,\overline{m})\Big|_{\xi = \mathcal{K}} = \Psi(\mathcal{K}, \eta)\Big|_{\eta = \overline{m}},
\end{align}
which is exactly \eqref{solvabilityhat}.

By applying \eqref{solvabilityhat},
we will prove that the unique existence of the weak solution $\Psi \in H^1(\overline{\Omega_+^{\mathcal{K}}})$.
By \eqref{solvabilityhat}, one can find a continuous function $\mathcal{G}_*$ of the following form:
\begin{align}
\mathcal{G}_*(\xi,\eta)
  \defs & \Psi(\mathcal{K},\eta) + \frac{\eta}{\overline{m}}\big( \Psi(\xi,\overline{m}) -  \Psi(\mathcal{K},\overline{m})\big) \notag\\
  &- \frac{\xi- \mathcal{K}}{L-\mathcal{K}} \Big(\frac{\eta}{\overline{m}} \big( \Psi(L,\overline{m})- \Psi(\mathcal{K},\overline{m})  \big) - \big(\Psi(L,\eta) - \Psi(\mathcal{K},\eta)\big) \Big).
\end{align}
Let
\begin{align}\label{widehatdef}
 \widehat{\Psi}(\xi,\eta)\defs {\Psi}(\xi,\eta) - \mathcal{G}_*(\xi,\eta).
\end{align}
Then \eqref{eq2s}-\eqref{brym=} become
\begin{align}\label{reeq2hat}
  &\partial_{\xi}\big(\frac{\bar{q}_+(\eta)}{ A_+^2(\eta) } \partial_{\xi} \widehat{\Psi} \big) + \partial_{\eta} \big(\frac{1}{A_+^2(\eta)} \frac{\bar{\rho}_+^2(\eta) \bar{q}_+^3(\eta)}{1-\overline{M}_+^2(\eta)}\partial_{\eta} \widehat{\Psi} \big) + \frac{g^2}{ \bar{q}_+^3(\eta)A_+^2(\eta)}
\widehat{\Psi}\notag\\
=& \mathcal{F}_2 - \frac{\bar{q}_+(\eta)}{ A_+^2(\eta) } \partial_{\xi}^2 \mathcal{G}_* - \partial_{\eta} \big(\frac{1}{A_+^2(\eta)} \frac{\bar{\rho}_+^2(\eta) \bar{q}_+^3(\eta)}{1-\overline{M}_+^2(\eta)}\partial_{\eta} \mathcal{G}_*\big) - \frac{g^2}{ \bar{q}_+^3(\eta)A_+^2(\eta)}
\mathcal{G}_*\notag\\
\defs & \widehat{F}(\xi, \eta),\quad \text{in}\quad {\Omega}_+^{\mathcal{K}},\\
&\widehat{\Psi}=0, \quad \text{on}\quad \partial{\Omega}_+^{\mathcal{K}}.\label{hatPhi=0}
\end{align}
Let $\zeta\in H_0^1({\Omega}_+^{\mathcal{K}})$ be the test function, $\widehat{\Psi}$ is a weak solution of the problem \eqref{reeq2hat}-\eqref{hatPhi=0} if
\begin{align}
  &- \int_{{\Omega}_+^{\mathcal{K}}}\Big( \big(  \frac{\bar{q}_+(\eta)}{ A_+^2(\eta) } \partial_{\xi} \widehat{\Psi}\big) \partial_{\xi} \zeta +  \big(\frac{1}{A_+^2(\eta)} \frac{\bar{\rho}_+^2(\eta) \bar{q}_+^3(\eta)}{1-\overline{M}_+^2(\eta)}\partial_{\eta} \widehat{\Psi} \big) \partial_{\eta}\zeta \Big)\dif \xi \dif \eta\notag\\
  & + \int_{{\Omega}_+^{\mathcal{K}}}\frac{g^2}{ \bar{q}_+^3(\eta)A_+^2(\eta)}
\widehat{\Psi} \zeta \dif \xi \dif \eta =  \int_{{\Omega}_+^{\mathcal{K}}} \widehat{F}\zeta \dif \xi \dif \eta.
\end{align}
For $\widehat{\Psi}, \zeta\in H_0^1({\Omega}_+^{\mathcal{K}})$,
let
\begin{align}
  \mathcal{B}[\widehat{\Psi}, \zeta] \defs& \int_{{\Omega}_+^{\mathcal{K}}}\Big( \big(  \frac{\bar{q}_+(\eta)}{ A_+^2(\eta) } \partial_{\xi} \widehat{\Psi}\big) \partial_{\xi} \zeta +  \big(\frac{1}{A_+^2(\eta)} \frac{\bar{\rho}_+^2(\eta) \bar{q}_+^3(\eta)}{1-\overline{M}_+^2(\eta)}\partial_{\eta} \widehat{\Psi}\big) \partial_{\eta}\zeta \Big)\dif \xi \dif \eta\notag\\
  & - \int_{{\Omega}_+^{\mathcal{K}}}\frac{g^2}{ \bar{q}_+^3(\eta)A_+^2(\eta)}
\widehat{\Psi}\zeta \dif \xi \dif \eta.
\end{align}
It is easy to check that $\mathcal{B}[\widehat{\Psi}, \zeta]$ is bounded. In addition, employing the Poincar\'{e} inequality (see \cite[Proposition 3.10]{GM}), we obtain
\begin{align}\label{coercive5D}
   \mathcal{B}[\widehat{\Psi}, \widehat{\Psi}] =& \int_{{\Omega}_+^{\mathcal{K}}}\Big( \frac{\bar{q}_+(\eta)}{ A_+^2(\eta) } (\partial_{\xi} \widehat{\Psi})^2+  \frac{1}{A_+^2(\eta)} \frac{\bar{\rho}_+^2(\eta) \bar{q}_+^3(\eta)}{1-\overline{M}_+^2(\eta)}(\partial_{\eta} \widehat{\Psi})^2 \Big)\dif \xi \dif \eta\notag\\
  & - \int_{{\Omega}_+^{\mathcal{K}}}\frac{g^2}{ \bar{q}_+^3(\eta)A_+^2(\eta)}
\widehat{\Psi}^2 \dif \xi \dif \eta\notag\\
\geq&  \min\limits_{\eta\in[0,\overline{m}]}\Big\{ \frac{\bar{q}_+(\eta)}{ A_+^2(\eta) } ,\,  \frac{1}{A_+^2(\eta)} \frac{\bar{\rho}_+^2(\eta) \bar{q}_+^3(\eta)}{1-\overline{M}_+^2(\eta)}    \Big\}\frac{1}{L^2}\int_{{\Omega}_+^{\mathcal{K}}}\widehat{\Psi}^2 \dif \xi \dif \eta\notag\\
  & - g^2 \max\limits_{\eta\in[0,\overline{m}]}\Big\{\frac{1}{ \bar{q}_+^3(\eta)A_+^2(\eta)}\Big\}\int_{{\Omega}_+^{\mathcal{K}}}
\widehat{\Psi}^2 \dif \xi \dif \eta\notag\\
\geq &\min\limits_{\eta\in[0,\overline{m}]}\Big\{ \frac{\bar{q}_+(\eta)}{ A_+^2(\eta) } ,\,  \frac{1}{A_+^2(\eta)} \frac{\bar{\rho}_+^2(\eta) \bar{q}_+^3(\eta)}{1-\overline{M}_+^2(\eta)}    \Big\}\frac{1}{2L^2}\int_{{\Omega}_+^{\mathcal{K}}}\widehat{\Psi}^2 \dif \xi \dif \eta,
\end{align}
where the last inequality is deduced by \eqref{assumpthm}.
Then one has $ \mathcal{B}[\widehat{\Psi}, \widehat{\Psi}]> \beta \|\widehat{\Psi}_+\|_{H_0^1(\mathcal{D})}^2$,
where the constant $\beta>0$ depends on $\overline{U}_+$ and $L$.
By the Lax-Milgram, there exists a unique solution $\widehat{\Psi}\in H_0^1({\Omega}_+^{\mathcal{K}})$ to the problem \eqref{reeq2hat}-\eqref{hatPhi=0}. Thus, applying $\Psi(\xi,\eta) =\widehat{\Psi}(\xi,\eta)+ \mathcal{G}_*(\xi,\eta)$, one can obtain the unique existence of the weak solution $\Psi\in H^1({\Omega}_+^{\mathcal{K}})$ to the problem \eqref{eq2s}-\eqref{brym=}.

Then we establish the $L^\infty$ estimate for $\widehat{\Psi}$. Let
\begin{align}
 L\widehat{\Psi}\defs  a_{11}(\eta)\partial_{\xi}^2 \widehat{\Psi} + a_{22}(\eta)\partial_{\eta}^2 \widehat{\Psi} + \partial_{\eta} a_{22}(\eta)\partial_{\eta} \widehat{\Psi} + a_0(\eta)
\widehat{\Psi}
\end{align}
be the left hand side of the equation \eqref{reeq2hat}, where
\begin{align}
  a_{11}(\eta) \defs& \frac{\bar{q}_+(\eta)}{ A_+^2(\eta) } >0, \quad a_{22}(\eta)\defs\frac{1}{A_+^2(\eta)} \frac{\bar{\rho}_+^2(\eta) \bar{q}_+^3(\eta)}{1-\overline{M}_+^2(\eta)}>0,\label{a1122>0}\\
a_0(\eta)\defs& \frac{g^2}{ \bar{q}_+^3(\eta)A_+^2(\eta)}>0.\label{a0>0}
\end{align}
Let
\begin{align}
  \vartheta(\xi,\eta) = a\big( 1+(\overline{m} -\eta)^2 + b\big( \xi^{\alpha} + (L-\xi)^{\alpha}  \big)  \big),
\end{align}
where the constant $0<\alpha <1$, the positive constants $a$ and $b$ will be determined later.
Direct calculations yield that
\begin{align}\label{Lvar=eq}
 L\vartheta =& a_{11}(\eta) \partial_{\xi}^2\vartheta + a_{22}(\eta)\partial_{\eta}^2 \vartheta + \partial_{\eta} a_{22}(\eta) \partial_{\eta}\vartheta+ a_0(\eta)\vartheta\notag\\
 =&-ab \alpha (1-\alpha)a_{11}(\eta)\big( \xi^{\alpha-2} + (L-\xi)^{\alpha-2} \big) + 2a a_{22}(\eta)\notag\\
 & -2a\partial_{\eta} a_{22}(\eta)(\overline{m} - \eta)+ a a_0(\eta)\big( 1+(\overline{m} -\eta)^2 + b\big( \xi^{\alpha} + (L-\xi)^{\alpha}  \big)  \big).
\end{align}
Note
\begin{align}
  \xi^{\alpha -2} + (L - \xi)^{\alpha-2}  > 2L^{\alpha-2},\quad \text{for}\quad 0<\alpha <1.
\end{align}
So \eqref{Lvar=eq} yields that
\begin{align}
  L\vartheta
 \leq & a \big( - b \alpha (1-\alpha)a_{11}(\eta)L^{\alpha-2} + 2 a_{22}(\eta) + 2|\partial_{\eta} a_{22}(\eta)|\overline{m}+ a_0(\eta)( 1+\overline{m}^2)\big)\notag\\
 &+ ab \big( -\alpha (1-\alpha)a_{11}(\eta)L^{\alpha-2} + 2a_0(\eta)L^{\alpha}\big).
\end{align}
Let
\begin{align}
  b= \frac{2\big( 2 \|a_{22}(\eta)\|_{L^\infty([0,\overline{m}])} + 2\overline{m}\|\partial_{\eta} a_{22}(\eta)\|_{L^\infty([0,\overline{m}])}+ ( 1+\overline{m}^2) \|a_0(\eta)\|_{L^\infty([0,\overline{m}])}\big)}{\alpha (1-\alpha)L^{\alpha-2} \min\limits_{\eta\in[0,\overline{m}]}a_{11}(\eta)}.
\end{align}
It is easy to check that
\begin{align}
- \frac12 b \alpha (1-\alpha)a_{11}(\eta)L^{\alpha-2} + 2 a_{22}(\eta) + 2|\partial_{\eta} a_{22}(\eta)|\overline{m}+ a_0(\eta)( 1+\overline{m}^2) < 0.\label{require1eq}
\end{align}
By \eqref{assumpthm}, one has
\begin{align}\label{require2eq}
 -\alpha (1-\alpha)a_{11}(\eta)L^{\alpha-2} + 2a_0(\eta)L^{\alpha} = \frac{2 L^{\alpha}\big(  g^2 -  \frac{ \alpha (1-\alpha)\bar{q}_+^4(\eta)}{ 2L^2}\big)}{ \bar{q}_+^3(\eta)A_+^2(\eta)} <0.
\end{align}
Thus
\begin{align}
L\vartheta\leq -\frac12 a b \alpha (1-\alpha)a_{11}(\eta)L^{\alpha-2}.
\end{align}
Let $a = C \|\widehat{F}\|_{L^\infty(\Omega_+^{\mathcal{K}})}$ for a large positive constant $C$, then one has
\begin{align}
  L\vartheta\leq -\frac12 a b \alpha (1-\alpha)a_{22}(\eta)\overline{m}^{\alpha -2} <  - \|\widehat{F}\|_{L^\infty(\Omega_+^{\mathcal{K}})}.
\end{align}
Therefore
 \begin{align}
   &L(\vartheta \pm \widehat{\Psi}_+) <0, \quad \text{in}\quad \Omega_+^{\mathcal{K}},\\
   &\vartheta \pm \widehat{\Psi}_+ > 0, \quad \text{on}\quad \partial\Omega_+^{\mathcal{K}}.
 \end{align}
 Applying the maximum principle, it follows that
 \begin{align}\label{interes}
   \|\widehat{\Psi}\|_{L^{\infty}(\Omega_+^{\mathcal{K}})} \leq C \vartheta \leq C\|\widehat{F}\|_{L^{\infty}(\Omega_+^{\mathcal{K}})},
 \end{align}
where the positive constant $C$ depends on $\overline{U}_+$.

Next, we consider the $L^\infty$ estimate for $\widehat{\Psi}$ near the corners. Without loss of generality, we only consider the estimate near the point $Q_4\defs (\mathcal{K},\overline{m})$, since the other corner points can be treated similarly. Let $B_{r_0}(Q_4)$ be the ball with the radius $r_0$ and the center $Q_4$, where $r_0$ is a small positive constant which will be determined to later.
Let
\begin{align}
  \xi - \mathcal{K} = r\cos \varpi, \quad \eta-\overline{m}  = r\sin \varpi,
\end{align}
where $ \varpi = \arctan\frac{\eta - \overline{m}}{\xi - \mathcal{K}} + 2\pi$ with $\varpi \in [\frac{3\pi}{2}, 2\pi]$ and $r\in(0, r_0)$.
Let the constant $\alpha \in (0,1)$ and
\begin{align}
   \upsilon (r,\varpi) = h r^{1+\alpha}\sin \big(\frac{3+\alpha}{2} (\varpi-\frac{3\pi}{2}) + \frac{1-\alpha}{8}\pi \big),
\end{align}
where $h$ is a large positive constant, which will be determined to later.
Direct calculations yield that
\begin{align}
\upsilon (r,\frac{3\pi}{2}) =& h r^{1+\alpha}\sin \big( \frac{1-\alpha}{8}\pi\big) >0,\\
 \upsilon (r,2\pi) =& h r^{1+\alpha}\sin \big( \frac{7+\alpha}{8}\pi\big) >0,\\
  \upsilon (r_0,\varpi) =& h r_0^{1+\alpha}\sin \big( \frac{3+\alpha}{2} (\varpi-\frac{3\pi}{2}) + \frac{1-\alpha}{8}\pi\big) \geq   h r_0^{1+\alpha} \sin\big( \frac{15+\alpha}{16}\pi \big).
\end{align}
Let
\begin{align}
   L_{\overline{m}} \upsilon \defs a_{11}(\overline{m}) \partial_{\xi}^2 \upsilon + a_{22}(\overline{m})\partial_{\eta}^2 \upsilon,
\end{align}
where
\begin{align}
    a_{11}(\overline{m}) = \frac{\bar{q}_+(\overline{m})}{ A_+^2(\overline{m}) },\quad a_{22}(\overline{m}) = \frac{1}{A_+^2(\overline{m})} \frac{\bar{\rho}_+^2(\overline{m}) \bar{q}_+^3(\overline{m})}{1-\overline{M}_+^2(\overline{m})}.
\end{align}
One may assume that $$ a_{11}(\overline{m}) = a_{22}(\overline{m})=1.$$ Otherwise, $L_{\overline{m}}$ can be changed into the Laplace operator by the transformation $(\tilde{\xi} , \tilde{\eta} ) = \big(\frac{\xi}{\sqrt{  a_{11}(\overline{m})}},  \frac{\eta}{\sqrt{ a_{22}(\overline{m})}}\big)$. Further calculations yield that
\begin{align}
   L\upsilon =& L_{\overline{m}}\upsilon +  \big( L \upsilon - L_{\overline{m}}\upsilon   \big)\notag\\
   \leq &  \big( \partial_{\xi}^2 \upsilon + \partial_{\eta}^2 \upsilon  \big) + Ch r^{\alpha -1}\big( | a_{11}(\eta) - a_{11}(\overline{m})| + | a_{22}(\eta) - a_{22}(\overline{m})|\big)\notag\\
    &+ C h \big( |\partial_{\eta}a_{22}(\eta)| r^{\alpha}  +   a_0(\eta) r^{1+\alpha}\big)\notag\\
    \leq & - h r^{\alpha -1}\frac{(1-\alpha)(5+3\alpha)}{4} \sin\big( \frac{15+\alpha}{16}\pi \big) \notag\\
    & + Ch r^{\alpha -1} r_0^{\alpha}+ C h\big(r^{\alpha}  + r^{\alpha +1}\big)\notag\\
    =& - h r^{\alpha -1} \big( \frac{(1-\alpha)(5+3\alpha)}{4} \sin\big( \frac{15+\alpha}{16}\pi \big) - Cr_0^{\alpha} - C(r_0+r_0^2)  \big)\notag\\
    \leq &  - \|\widehat{F}\|_{L^\infty(\Omega_+^{\mathcal{K}})},
\end{align}
where the positive constant $C$ depends on $\overline{U}_+$, and the small constant $r_0>0$ depends on $\alpha$ and $C$.
By applying \eqref{interes},
we can choose $h = C r_0^{-1-\alpha} \|\widehat{F}\|_{L^\infty(\Omega_+^{\mathcal{K}})}$, where $C$ is a large positive constant such that
\begin{align}
  &L(\upsilon \pm \widehat{\Psi}) <0, \quad \text{in}\quad B_{r_0}(Q_4)\cap \Omega_+^{\mathcal{K}},\\
   &\upsilon \pm \widehat{\Psi} > 0, \quad \text{on}\quad \partial(B_{r_0}(Q_4)\cap \Omega_+^{\mathcal{K}}) \cap \partial\Omega_+^{\mathcal{K}},\\
   & \upsilon \pm \widehat{\Psi} > 0, \quad \text{on}\quad \partial(B_{r_0}(Q_4)\cap \Omega_+^{\mathcal{K}}) \cap \{ r = r_0\}.
\end{align}
It follows from the maximum principle that
\begin{align}\label{PsiLinftycorner}
  \|\widehat{\Psi}\|_{L^\infty(B_{r_0}(Q_4)\cap \Omega_+^{\mathcal{K}})} \leq C r^{1+\alpha}\|\widehat{F}\|_{L^\infty(\Omega_+^{\mathcal{K}})}.
\end{align}
Employing \eqref{interes}, \eqref{PsiLinftycorner}, Schauder interior and boundary estimates in \cite[Chapter 8]{GT}, and the definition of $\widehat{F}$ in \eqref{reeq2hat}, one can obtain global $C^{1,\alpha}$ estimate for $\widehat{\Psi}$, \emph{i.e.}
\begin{align}\label{widehatPsi1alpha}
\|\widehat{\Psi}\|_{C^{1,\alpha}(\overline{\Omega_+^{\mathcal{K}}})} \leq C \|\widehat{F}\|_{0,\alpha;\Omega_+^{\mathcal{K}}}^{(1-\alpha;\{\overline{\Gamma_0^{\mathcal{K}}}\cup \overline{\Gamma_{\overline{m}}^{\mathcal{K}}}\})}.
\end{align}

Now we raise the regularity of $\widehat{\Psi}$ up to the weighted $C^{2, \alpha}$ norm. We only consider the solution $\widehat{\Psi}$ near point $Q_4\defs (\mathcal{K},\overline{m})$, since the other corner points $Q_i,(i=1,2,3)$ can be treated similarly.
For $0< R < \frac{L}{2}$, let
\begin{align}
  B_{R}^+ \defs B_R(Q_4) \cap \overline{\Omega_+^{\mathcal{K}}},\quad   \Gamma_{\overline{m}}^{(R)}\defs B_{R}^+ \cap \Gamma_{\overline{m}}^{\mathcal{K}},\quad
   \Gamma_s^{(R)}\defs B_{R}^+ \cap \Gamma_s^{\mathcal{K}}.
\end{align}
For any point $Q= (\xi,\eta)\in  B_{\frac{R}{2}}^+$, let $ d_{Q}\defs \dist (Q, Q_4)$, there will be at least one of the following three cases:
\begin{align}
  \text{(i)}\quad&  B_{\frac{d_Q}{10\mathcal{L}}}(Q) \subset B_R^+,\\
  \text{(ii)}\quad& Q \in  B_{\frac{d_{\widehat{Q}}}{2\mathcal{L}}}(\widehat{Q}), \quad \text{for}\quad \widehat{Q} \in \Gamma_{\overline{m}}^{(R)},\\
  \text{(iii)} \quad& Q \in  B_{\frac{d_{\widehat{Q}}}{2\mathcal{L}}}(\widehat{Q}), \quad \text{for }\quad \widehat{Q} \in \Gamma_s^{(R)},
\end{align}
for some constant $\mathcal{L}>1$.
Thus, it suffices to make the $C^{2,\alpha}$ estimates of $\widehat{\Psi}$ in the following domains:
\begin{align}
  \text{(I)}\quad&  B_{\frac{d_{Q}}{20\mathcal{L}}}(Q)\quad \text{when}\quad  B_{\frac{d_{Q}}{10\mathcal{L}}}(Q) \subset B_R^+,\\
  \text{(II)}\quad& B_{\frac{d_{\widehat{Q}}}{2\mathcal{L}}}(\widehat{Q}) \cap B_R^+\quad  \text{for}\quad \widehat{Q}\in  \Gamma_{\overline{m}}^{(\frac{R}{2})},\\
  \text{(III)}\quad& B_{\frac{d_{\widehat{Q}}}{2\mathcal{L}}}(\widehat{Q}) \cap B_R^+\quad  \text{for}\quad \widehat{Q}\in  \Gamma_s^{(\frac{R}{2})}.
\end{align}
We only discuss case $\text{(III)}$, and the other cases can be treated similarly. Let $\widehat{d}\defs \frac{d_{\widehat{Q}}}{2\mathcal{L}}$ and $\widehat{Q} = (\hat{\xi}, \hat{\eta})$. Let
\vspace{-0.3cm}
\begin{align}
 \mathbf{Z}\defs (Z_1, Z_2) \defs \frac{(\xi - \hat{\xi}, \eta - \hat{\eta})}{\widehat{d}},
\end{align}
then the domain $B_{\frac{d_{\widehat{Q}}}{2\mathcal{L}}}(\widehat{Q}) =  B_{\widehat{d}}(\widehat{Q})$ becomes $B_1(0,0)$. In addition, the shock front $\xi =\mathcal{K}$ becomes $Z_1 = \frac{\mathcal{K} - \hat{\xi}}{\widehat{d}}$. Then we consider the problem in the domain
\begin{align}
 \Omega_1^{\widehat{Q}}\defs B_1(0,0) \cap \big\{Z_1 > \frac{\mathcal{K} - \hat{\xi}}{\widehat{d}} \big\}.
\end{align}
Define
\begin{align}
  \Psi^{\aleph}(Z_1, Z_2) \defs \frac{\widehat{\Psi}(\hat{\xi} +\widehat{d} Z_1,\hat{\eta} +\widehat{d} Z_2)}{\widehat{d}^{1+\alpha}}, \quad F^{\aleph}(Z_1, Z_2)\defs \frac{\widehat{F}(\hat{\xi} +\widehat{d} Z_1,\hat{\eta} +\widehat{d} Z_2)}{\widehat{d}^{\alpha-1}}.
\end{align}
By \eqref{reeq2hat}, one has
\begin{align}
&a_{11}(\eta)  \Psi_{Z_1Z_1}^{\aleph} + a_{22}(\eta) \Psi_{Z_2Z_2}^{\aleph}+ (\partial_{\eta} a_{22}(\eta)) \widehat{d} \Psi_{Z_2}^{\aleph} + a_0(\eta)
\widehat{d}^2 \Psi^{\aleph}  = F^\aleph ,\,\text{in}\quad \Omega_1^{\widehat{Q}},\\
&\Psi^{\aleph} =0, \quad \text{on}\quad \partial  \Omega_1^{\widehat{Q}} \cap \big\{Z_1 = \frac{\mathcal{K} - \hat{\xi}}{\widehat{d}} \big\}.
\end{align}
Then by standard local estimates for linear elliptic boundary problem and \eqref{widehatPsi1alpha}, one has
\begin{align}
  \|\Psi^{\aleph}\|_{C^{2,\alpha}(\Omega_{\frac12}^{\widehat{Q}})} \leq C \big(  \|\Psi^{\aleph}\|_{C^0(\Omega_1^{\widehat{Q}})} + \|F^\aleph\|_{C^{0,\alpha}(\Omega_1^{\widehat{Q}}) }  \big)\leq  C \|\widehat{F}\|_{0,\alpha;\Omega_1^{\widehat{Q}} }^{(1-\alpha; Q_4)}.
\end{align}
Furthermore, one has
\begin{align}\label{weightes}
\|\widehat{\Psi}\|_{2,\alpha;B_{\frac{R}{2}}^+}^{(-1-\alpha;Q_4)} \leq
 C \|\widehat{F}\|_{0,\alpha;B_R^+}^{(1-\alpha; Q_4)}.
\end{align}
Then we modify the domain $B_R^+(Q_4)$ by smoothing out the corner at the point $Q_4$ using a mollification process. We denote the resulting domain as $D^+$. More precisely, $D^+$ represents an open domain that satisfies
\begin{align}
    D^+ \subset B_R^+(Q_4),\quad D^+ \backslash
   B_{\frac{R}{10}}^+(Q_4) = B_R^+(Q_4)\backslash
   B_{\frac{R}{10}}^+(Q_4),
\end{align}
and
\vspace{-0.3cm}
\begin{align}
   \partial D^+ \cap B_{\frac{R}{5}} (Q_4),\quad \text{is a }C^{2,\alpha}-\text{curve}.
\end{align}
In the smooth domain $D^+$, one can use the Schauder interior and boundary estimates in Chapter 8 of \cite{GT}, and obtain
\begin{align}
    \|\widehat{\Psi}\|_{C^{2,\alpha}(\overline{D^+})} \leq C \|\widehat{F}\|_{C^{0,\alpha}(\Omega_+^{\mathcal{K}})}.
\end{align}
Furthermore, one can obtain
\begin{align}\label{hatPsiWN}
\|\widehat{\Psi}\|_{2,\alpha; \Omega_+^{\mathcal{K}}}^{(-1-\alpha;\{Q_i\})} \leq C \|\widehat{F}\|_{0,\alpha;\Omega_+^{\mathcal{K}}}^{(1-\alpha;\{\overline{\Gamma_0^{\mathcal{K}}}\cup \overline{\Gamma_{\overline{m}}^{\mathcal{K}}}\})}.
\end{align}
By \eqref{hatPsiWN}, \eqref{widehatdef} and \eqref{NPsi}, one has \eqref{pthetaes}.
Thus, there exists a unique solution $(\widehat{\delta p_+},\, \widehat{\delta \theta_+})\in C_{1,\alpha}^{(-\alpha;\{Q_i\})}(\Omega_+^{\mathcal{K}})^2 $ to the problem \eqref{NLdoteq1}-\eqref{Neq:theta4+re} if and only if \eqref{solvabilityhat} holds.
 In addition,
  \begin{align}\label{pthetaes}
&\|\widehat{\delta p_+}\|_{1,\alpha;\Omega_+^{\mathcal{K}}}^{(-\alpha;\{Q_i\})}  +  \|\widehat{\delta \theta_+}\|_{1,\alpha;\Omega_+^{\mathcal{K}}}^{(-\alpha;\{Q_i\})}  \notag\\
\leq & C\Big(\sum_{i=1}^2 \|F_i \|_{0,\alpha;\Omega_+^{\mathcal{K}}}^{(1-\alpha;\{\overline{\Gamma_0^{\mathcal{K}}}\cup \overline{\Gamma_{\overline{m}}^{\mathcal{K}}}\})}  + \sigma \|p_{ex}\|_{1,\alpha;\Gamma_{ex}}^{(-\alpha;\{Q_2, Q_3\})} + \sigma \|\Theta\|_{C^{2,\alpha}(\Gamma_{\overline{m}}^{\mathcal{K}})} + \|h_1\|_{1,\alpha;\Gamma_s^{\mathcal{K}}}^{(-\alpha;\{Q_1, Q_4\})} \Big),
\end{align}
where the constant $C$ depends on $\overline{U}_{\pm}$, $L$ and $\alpha$.

Therefore, \eqref{S+=eq}-\eqref{q+=eq}, \eqref{pthetasharp} and \eqref{pthetaes} yield that \eqref{pthetaqses}.
\end{proof}

\section{Nonlinear free boundary problem}
In order to solve the nonlinear free boundary problem in the sense of Definition \ref{RSP}, we will design an iteration scheme such that the assumptions \eqref{Assumpf}-\eqref{assumebry} hold. Then, the Banach fixed point theorem will applied to establish the well-posedness of this nonlinear free boundary  problem.

\subsection{Linearized problem and approximate shock position}
The first step is to introduce a linearized problem in
$\Omega_+^{\mathcal{K}_0}$ by defining
 \begin{align}
 &\delta {U}_+(\xi,\eta)\defs\delta U_+^{(0)}(\xi,\eta) = \big( \delta p_+^{(0)}, \, \delta\theta_+^{(0)},\, \delta q_+^{(0)},\, \delta S_+^{(0)}\big)^{\top}(\xi,\eta),\label{letUpsi}\\
 &\psi(\eta) \equiv \mathcal{K}\defs \mathcal{K}_0,\,\, F_1 \defs 0, \,\, F_3 \defs 0,\label{letUpsix}\\
 &F_2\defs \frac{g}{\bar{q}_+(\eta)A_+(\eta)}\big(\frac{m-\overline{m}}{\overline{m}}+ \frac{\delta p_+^{(0)}(\mathcal{K}_0, \eta) + \bar{\rho}_+(\eta)\bar{q}_+(\eta) \delta q_+^{(0)} (\mathcal{K}_0, \eta)}{\bar{\rho}_+(\eta) \bar{q}_+^2(\eta)} \big),\\
&\delta p_+^{(0)} (L, \eta) \defs \sigma p_{ex} (\int_{0}^{\eta}\frac{1}{\bar{\rho}_+(t) \bar{q}_+(t)} \dif t),\,\,  \mathcal{H}_j^{\sharp}\defs -{\mathbf{\alpha}}_{j-} \cdot {U}_-^{(0)},(j = 1,2,3,4),\label{letGj=0}
\end{align}
where
\begin{align}
  &{\mathbf{\alpha}}_1^- = -\frac{[\bar{p}]}{\bar{\rho}_- \bar{q}_-}\big(-\frac{1}{\bar{\rho}_- \bar{c}_-^2},\, 0,\, -\frac{1}{\bar{q}_-},\, \frac{1}{\gamma c_v}\big)^{\top},\label{alpha1-}\\
&{\mathbf{\alpha}}_2^- = -\frac{[\bar{p}]}{\bar{\rho}_- \bar{q}_-} \big(1 -\frac{\bar{p}_-}{\bar{\rho}_- \bar{c}_-^2},\, 0,\, \bar{\rho}_- \bar{q}_- -\frac{\bar{p}_-}{\bar{q}_-},\,\frac{\bar{p}_-}{\gamma c_v}\big)^{\top},\\
&{\mathbf{\alpha}}_3^- = -\big(\frac{1}{\bar{\rho}_-},\,0,\, \bar{q}_-, \, \frac{1}{(\gamma -1)c_v} \frac{\bar{p}_-}{\bar{\rho}_-}\big)^{\top},\,\, {\mathbf{\alpha}}_4^- = -(\bar{q}_-,\, 0,\,0 ,\, 0)^{\top}.\label{alpha4-}
\end{align}
In \eqref{letGj=0}, $U_-^{(0)}=\bar{U}_-+\delta U_-^{(0)}$ with $\delta U_-^{(0)}=\big( \delta p_-^{(0)}, \, \delta\theta_-^{(0)},\, \delta q_-^{(0)},\, \delta S_-^{(0)}\big)^{\top}$ satisfying the following equations in $\Omega_-^{\mathcal{K}_0}$:
\begin{align}
&\partial_{\xi} \big({\overline{M}_-^2(\eta)}\delta p_-^{(0)} +\bar{\rho}_-(\eta) \bar{q}_-(\eta)\delta q_-^{(0)}\big) + \bar{\rho}_-^2(\eta) \bar{q}_-^3(\eta)\partial_{\eta}\delta \theta_-^{(0)} =0,\label{Ldoteq1-}\\
 & \partial_{\xi} \delta\theta_-^{(0)} + \frac{1}{\bar{q}_-(\eta)}\partial_{\eta}\delta p_-^{(0)} - \frac{g}{\bar{q}_-^3(\eta)} \delta q_-^{(0)} = \frac{m-\overline{m}}{\overline{m}}\frac{g}{\bar{q}_-^2(\eta)},\label{Ldoteq2-}\\
 &\partial_{\xi} \big(\bar{\rho}_-(\eta)\bar{q}_-(\eta)\delta q_-^{(0)} + \delta p_-^{(0)}\big) + g\bar{\rho}_-(\eta)\delta \theta_-^{(0)} =0,\label{Ldoteq3-}\\
 &\partial_{\xi} \delta S_-^{(0)} =0, \label{Ldoteq4-}
\end{align}
and the following initial-boundary value conditions:
\begin{align}\label{dotU-0}
\delta U_-^{(0)}(0,\eta)=(\sigma p_{en}(\eta), 0, 0,0)^\top,\,\,
\delta \theta_-^{(0)}(\xi,0) = 0, \,\,
\delta \theta_-^{(0)}(\xi, \overline{m}) = \sigma\Theta(\xi).
\end{align}
From the standard theorems for two-dimensional hyperbolic equations (see \cite{LY1985}), it is easy to check that the solutions of problem \eqref{Ldoteq1-}-\eqref{dotU-0} satisfy
\begin{align}
&\frac{1}{\bar{\rho}_-(\eta)\bar{q}_-^2(\eta)} \big(\overline{M}_-^2(\eta) \delta p_-^{(0)} + {\bar{\rho}_-(\eta)\bar{q}_-(\eta)} \delta q_-^{(0)}\big)\notag\\
= & \frac{\overline{M}_-^2(\eta)}{\bar{\rho}_-(\eta)\bar{q}_-^2(\eta)} \sigma p_{en} (\eta)
  -{\bar{\rho}_-(\eta)\bar{q}_-(\eta)}\int_0^{\xi} \partial_\eta \delta \theta_-^{(0)}(s,\eta) \dif s,\label{dotrh--}\\
&(\delta p_-^{(0)} +  \bar{\rho}_-(\eta) \bar{q}_-(\eta) \delta q_-^{(0)}) (\xi,\eta) =  \sigma p_{en}(\eta) -  g\bar{\rho}_- (\eta) \int_0^\xi \delta \theta_-^{(0)} (s,\eta)\dif s,\label{dotrh-}\\
&\|\delta U_-^{(0)}\|_{C^{2,\alpha}( \overline{\Omega_-^{\mathcal{K}_0}})} \leq C\sigma\big(\|\Theta\|_{C^{2,\alpha}(\Gamma_{\overline{m}}\cap \overline{\Omega_-^{\mathcal{K}_0}} )} + \|p_{en}\|_{C^{2,\alpha}(\Gamma_{en})} \big) \leq C_{\overline{U}_-} \sigma,\label{dotU-es}\\
&\| U- \overline{U}_- - \delta U_-^{(0)}\|_{C^{1,\alpha}( \overline{\Omega_-^{\mathcal{K}_0}})} \leq C_- \sigma^2,\label{itera-}
\end{align}
where the positive constants $C_{\overline{U}_-} $ and $C_-$ depend on $\overline{U}_-$, $L$ and $\alpha$.
Employing \eqref{As-1H}-\eqref{h4==}, \eqref{alpha1-}-\eqref{alpha4-} and \eqref{dotrh--}-\eqref{dotrh-}, further tedious calculations yield that on the shock $\xi=L$,
\begin{align}\label{As-1dot}
\big(\delta p_+^{(0)},\,\delta q_+^{(0)},\, \delta S_+^{(0)},\, (\psi^{(0)})'\big)\defs \big(\frac{\bar{\rho}_+\bar{q}_+^2 }{1 - \overline{M}_+^2}h_1^{(0)},\, \frac{1}{\bar{\rho}_+ \bar{q}_+} h_2^{(0)},\,h_3^{(0)},\, h_4^{(0)} \big)(\eta),
\end{align}
where
\begin{align}
h_1^{(0)} =&\Big(\frac{1}{\bar{\rho}_+\bar{q}_+^2 } -  \frac{\gamma-1}{\gamma\bar{p}_+} (\frac{\bar{\rho}_+}{\bar{\rho}_-} -1 )- \frac{\overline{M}_-^2}{\bar{\rho}_-\bar{q}_-^2}\big( 1- ( \frac{1}{\bar{\rho}_+\bar{q}_+^2 }+\frac{\gamma-1}{\gamma \bar{p}_+})[\bar{p}] \big)   \Big)\sigma p_{en}(\eta)\notag\\
  & -\big(\frac{1}{\bar{q}_-\bar{q}_+ } -  \frac{\gamma-1}{\gamma\bar{p}_+} (\bar{\rho}_+ - \bar{\rho}_- )\big)g\int_0^{\mathcal{K}_0} \delta\theta_-^{(0)} (\xi,\eta)\dif \tau\notag\\
  & + \big(1- ( \frac{1}{\bar{\rho}_+\bar{q}_+^2 }+\frac{\gamma-1}{\gamma \bar{p}_+})[\bar{p}]
\big){\bar{\rho}_-\bar{q}_-}\int_0^{\mathcal{K}_0} \partial_\eta \delta\theta_-^{(0)} (\xi,\eta) \dif \xi,\label{LLG10}\\
h_2^{(0)} =&\big( 1 - \frac{\bar{\rho}_+\bar{q}_+^2}{1 - \overline{M}_+^2} ( \frac{1}{\bar{\rho}_+\bar{q}_+^2}- \frac{\gamma-1}{\gamma\bar{p}_+}(\frac{\bar{\rho}_+}{\bar{\rho}_-}-1)  ) \big) \big( \sigma p_{en}(\eta) -  g\bar{\rho}_- \int_0^{\mathcal{K}_0} \delta\theta_-^{(0)} (\xi,\eta)\dif \xi  \big)\notag\\
 & + \big( [\bar{p}]+ \frac{\bar{\rho}_+\bar{q}_+^2}{1 - \overline{M}_+^2} ( 1 - (\frac{1}{\bar{\rho}_+\bar{q}_+^2}+\frac{\gamma-1}{\gamma \bar{p}_+} ) [\bar{p}] ) \big) \notag\\
 &\quad \times \big( \frac{\overline{M}_-^2}{\bar{\rho}_-\bar{q}_-^2} \sigma p_{en} (\eta)
  -{\bar{\rho}_-\bar{q}_-}\int_0^{\mathcal{K}_0} \partial_\eta \delta\theta_-^{(0)}(\xi,\eta) \dif \xi \big),\\
h_3^{(0)} =&   \frac{(\gamma-1)c_v}{\bar{p}_+}\big(\frac{\bar{\rho}_+}{\bar{\rho}_-} -1 -\frac{\overline{M}_-^2}{\bar{\rho}_- \bar{q}_-^2}[\bar{p}]\big)\sigma p_{en}(\eta)- \frac{(\gamma-1)c_v}{\bar{p}_+}\notag\\
& \quad \times \big((\bar{\rho}_+ - \bar{\rho}_-) g \int_0^{\mathcal{K}_0} \delta\theta_-^{(0)} (\xi,\eta)\dif \xi - \bar{\rho}_-\bar{q}_-[\bar{p}]\int_{0}^{\mathcal{K}_0} \partial_{\eta}\delta \theta_-^{(0)}(\xi,\eta)\dif \xi \big), \label{LLGRHs}\\
h_4^{(0)} =& \frac{\bar{q}_+ \delta\theta_+^{(0)}  - \bar{q}_- \delta \theta_-^{(0)}}{[\bar{p}]} \frac{\overline{m}}{m}.\label{dotH4=}
 \end{align}

Based on Theorem \ref{Rthm}, we have the following lemma.
\begin{lem}\label{lem(0)=}
Assume that \eqref{assumpthm}  holds. 
There exists a unique solution $(\delta p_+^{(0)}, \, \delta\theta_+^{(0)},\, \delta q_+^{(0)},\, \delta S_+^{(0)}; \mathcal{K}_0)$ for the problem in the sense of Definition \ref{RSP} with \eqref{letUpsix}-\eqref{letGj=0},
if and only if
\begin{align}\label{dotsolvability}
  &\int_0^{\overline{m}}\frac{A_+(\eta)}{\bar{\rho}_+(\eta) \bar{q}_+(\eta)} h_1^{(0)} (\mathcal{K}_0, \eta)\dif \eta +  A_+(\eta) \int_{\mathcal{K}_0}^L \sigma \Theta(\xi)\dif \xi \notag\\
  =& \int_0^{\overline{m}} \frac{1-\overline{M}_+^2(\eta)}{\bar{\rho}_+^2(\eta) \bar{q}_+^3(\eta)} A_+(\eta) \sigma p_{ex}(\int_{0}^{\eta}\frac{1}{\bar{\rho}_+(t) \bar{q}_+(t)} \dif t)\dif \eta.
\end{align}
In addition, the solution, if it exists, satisfies
  \begin{align}\label{pthetaqses}
&\|(\delta p_+^{(0)}, \, \delta\theta_+^{(0)}) \|_{1,\alpha;\Omega_+^{\mathcal{K}_0}}^{(-\alpha;\{Q_i\})} +  \|(\delta q_+^{(0)},\, \delta S_+^{(0)})\|_{1,\alpha;\Omega_+^{\mathcal{K}_0}}^{(-\alpha;\{\overline{\Gamma_0^{\mathcal{K}}}\cup \overline{\Gamma_{\overline{m}}^{\mathcal{K}}}\})}\notag\\
\leq & C\sigma \big(\|p_{en}\|_{C^{2,\alpha}(\Gamma_{en})} +  \|p_{ex}\|_{C^{2,\alpha}(\Gamma_{ex})} + \|\Theta\|_{C^{2,\alpha}([0,L])}\big),
\end{align}
where the constant $C$ depends on $\overline{U}_{\pm}$, $L$ and $\alpha$.
\end{lem}

Next, we will find a unique $\mathcal{K}_0$ such that \eqref{dotsolvability} holds.

Substituting $h_1^{(0)}$ given in \eqref{LLG10} into \eqref{dotsolvability}, one has
\begin{align}\label{dotsolvability1}
& - \int_0^{\overline{m}}\int_0^{\mathcal{K}_0} K_1(\eta) \delta\theta_-^{(0)} (\xi,\eta)\dif \xi
\dif \eta + \int_0^{\overline{m}}\int_0^{\mathcal{K}_0}K_2(\eta) \partial_\eta \delta\theta_-^{(0)} (\xi,\eta) \dif \xi
\dif \eta \notag\\
   &+  A_+(\eta) \int_{\mathcal{K}_0}^L \sigma \Theta(\xi)\dif \xi \notag\\
  =& \sigma \int_0^{\overline{m}} K_3(\eta) p_{ex}(\int_{0}^{\eta}\frac{1}{\bar{\rho}_+(t) \bar{q}_+(t)} \dif t)\dif \eta -\sigma \int_0^{\overline{m}}
  K_4(\eta) p_{en}(\eta)\dif \eta,
\end{align}
where
\begin{align}
 K_1(\eta)\defs& \frac{gA_+(\eta) }{\bar{\rho}_+(\eta)\bar{q}_+(\eta)}\big(\frac{1}{\bar{q}_-(\eta)\bar{q}_+(\eta)} -  \frac{\gamma-1}{\gamma\bar{p}_+(\eta)} (\bar{\rho}_+(\eta) - \bar{\rho}_-(\eta) )\big),\label{K1eta=}\\
  K_2(\eta)\defs& A_+(\eta)\big(1- ( \frac{1}{\bar{\rho}_+(\eta)\bar{q}_+^2(\eta) }+\frac{\gamma-1}{\gamma \bar{p}_+(\eta)})[\bar{p}(\eta)]
\big),\\
 K_3(\eta)\defs& \frac{1- \bar{M}_+^2(\eta)}{\bar{\rho}_+^2(\eta)\bar{q}_+^3(\eta)}A_+(\eta),\\
K_4(\eta)\defs& \frac{-A_+(\eta)}{\bar{\rho}_+(\eta)\bar{q}_+(\eta)}
\Big(\frac{1}{\bar{\rho}_+(\eta)\bar{q}_+^2 (\eta)} -  \frac{\gamma-1}{\gamma\bar{p}_+(\eta)} \big(\frac{\bar{\rho}_+(\eta)} {\bar{\rho}_-(\eta)} -1 \big)\notag\\
&\qquad \qquad\qquad - \frac{\overline{M}_-^2(\eta)}{\bar{\rho}_-(\eta)\bar{q}_-^2(\eta)}
\big( 1- \big( \frac{1}{\bar{\rho}_+(\eta)\bar{q}_+^2 (\eta) }+\frac{\gamma-1}{\gamma \bar{p}_+(\eta)}\big)[\bar{p}(\eta)] \big)   \Big).\label{K4eta=}
\end{align}
Then it follows from \eqref{dotU-0} that, 
\begin{align}\label{dotsolvability2}
  & - \int_0^{\overline{m}}\int_0^{\mathcal{K}_0} \big(K_1(\eta)  + K_2'(\eta) \big) \delta\theta_-^{(0)} (\xi,\eta)\dif \xi
\dif \eta\notag\\
& + \sigma K_2(\eta) \int_0^{\mathcal{K}_0}\Theta(\xi)\dif \xi +  \sigma A_+(\eta) \int_{\mathcal{K}_0}^L \Theta(\xi)\dif \xi \notag\\
  =& \sigma \int_0^{\overline{m}} K_3(\eta) p_{ex}(\int_{0}^{\eta}\frac{1}{\bar{\rho}_+(t) \bar{q}_+(t)} \dif t)\dif \eta -\sigma \int_0^{\overline{m}}
  K_4(\eta) p_{en}(\eta)
\dif \eta.
\end{align}
Let $K(\eta) \defs K_1(\eta)  + K_2'(\eta) $ and
\begin{align}
  J_1(\xi)\defs &- \int_0^{\overline{m}}\int_0^{\xi} K(\eta) \frac{\delta\theta_-^{(0)}(s,\eta)}{\sigma} \dif s
\dif \eta\notag\\
& +  K_2(\eta) \int_0^{\xi}\Theta(s)\dif s +   A_+(\eta) \int_{\xi}^L \Theta(s)\dif s,\label{J1=eq}\\
J_2\defs & \int_0^{\overline{m}} K_3(\eta) p_{ex}(\int_{0}^{\eta}\frac{1}{\bar{\rho}_+(t) \bar{q}_+(t)} \dif t)\dif \eta - \int_0^{\overline{m}}
  K_4(\eta) p_{en}(\eta)
\dif \eta.
\end{align}

So to show \eqref{dotsolvability} holds is equivalent to show
\begin{align}\label{eq2:J1K0=J2}
  J_1(\mathcal{K}_0) =  J_2.
\end{align}
Then we have the following lemma.
\begin{lem}\label{determinemathcalK0}
 Assume
\begin{align}
&\int_0^{\overline{m}}\frac{K(\eta)}{ \bar{q}_-(\eta) }\big(p_{en}'(\eta)  - \frac{m-\overline{m}}{\sigma \overline{m}}\frac{g}{\bar{q}_-(\eta)}\big)\dif \eta >0,\label{assump1}\\
&0<  L_0< \min\{L_*,\,\,  L \},\label{L0lessmin}
\end{align}
where
$$
L_*\defs \frac{2}{\mathcal{I}}\int_0^{\overline{m}}\frac{K(\eta)}{ \bar{q}_-(\eta) }\big(p_{en}'(\eta)  - \frac{m-\overline{m}}{\sigma \overline{m}}\frac{g}{\bar{q}_-(\eta)}\big)\dif \eta
$$
and $\mathcal{I} \defs C_{\overline{U}_-}\int_0^{\overline{m}}|K(\eta)|
\dif \eta+ \|K_2(\eta) - A_+(\eta)\|_{L^\infty([0,\overline{m}])} \|\Theta''\|_{L^\infty([0,L])}$.
If
\begin{align}\label{eq:<J2<}
  J_1(0) < J_2 <  J_1(0) + \widehat{J}_1(L_0),
\end{align}
where $J_1(0) \defs A_+(\eta) \int_{0}^L \Theta(\tau) \dif \tau$ and
\begin{align*}
 \widehat{J}_1(L_0)\defs &  \frac{L_0^2}{6} \Big( 3\int_0^{\overline{m}}\frac{K(\eta)}{ \bar{q}_-(\eta) }\big( p_{en}'(\eta)  - \frac{m-\overline{m}}{\sigma \overline{m}}\frac{g}{\bar{q}_-(\eta)}\big)\dif \eta -  L_0\mathcal{I}\Big)>0,
\end{align*}
there exists a unique $\mathcal{K}_0\in (0,L_0)$ such that \eqref{eq2:J1K0=J2} holds.
\end{lem}

\begin{proof}
By \eqref{J1=eq}, one has
\begin{align}\label{J1'=}
  J_1'(\xi) =& - \int_0^{\overline{m}}K(\eta) \frac{\delta\theta_-^{(0)} (\xi,\eta)}{\sigma}
\dif \eta +  K_2(\eta) \Theta(\xi) - A_+(\eta) \Theta(\xi).
\end{align}
Applying $\Theta(0)=0$ and $\theta_-^{(0)} (0,\eta)=0$, it is easy to see that
\begin{align}
J_1(0) = A_+(\eta) \int_{0}^L \Theta(\xi) \dif \xi,\,
\quad J_1'(0) =0.
\end{align}
In addition, \eqref{Ldoteq2-} yields that
\begin{align}
&- \partial_{\xi}\int_0^{\overline{m}}K(\eta) \delta\theta_-^{(0)} (\xi,\eta)
\dif \eta\notag\\
=& \int_0^{\overline{m}}K(\eta)
\big(\frac{1}{\bar{q}_-(\eta)}\partial_{\eta}\delta p_-^{(0)}  - \frac{g}{\bar{q}_-^3(\eta)} \delta q_-^{(0)}  - \frac{m-\overline{m}}{\overline{m}}\frac{g}{\bar{q}_-^2(\eta)}\big)\dif \eta,
\end{align}
which implies that
\begin{align}\label{J1xi=}
  J_1''(\xi) =& - \partial_{\xi}\int_0^{\overline{m}}K(\eta) \delta \theta_-^{(0)} (\xi,\eta)
\dif \eta + K_2(\eta) \Theta'(\xi) - A_+(\eta) \Theta'(\xi)\notag\\
=&\int_0^{\overline{m}}K(\eta)\frac{1}{\sigma}\big(\frac{1}{\bar{q}_-(\eta)}\partial_{\eta} \delta p_-^{(0)}  - \frac{g}{\bar{q}_-^3(\eta)}\delta q_-^{(0)}  - \frac{m-\overline{m}}{\overline{m}}\frac{g}{\bar{q}_-^2(\eta)}\big)\dif \eta\notag\\
&  +  \big(K_2(\eta) - A_+(\eta)\big)\Theta'(\xi).
\end{align}
Employing $\Theta'(0) =0$ and \eqref{dotU-0}, \eqref{J1xi=} yields that
\begin{align}
J_1''(0) = \int_0^{\overline{m}}\frac{K(\eta)}{ \bar{q}_-(\eta) }\big(p_{en}'(\eta)  - \frac{m-\overline{m}}{\sigma \overline{m}}\frac{g}{\bar{q}_-(\eta)}\big)\dif \eta.
\end{align}
By applying \eqref{assump1}, one has
\begin{align}\label{I'xi>}
J_1'(\xi)=&
J_1''(0) {\xi} + \frac12 J_1'''(\hat{\xi}) {\xi}^2\notag\\
  =& {\xi} \int_0^{\overline{m}}\frac{K(\eta)}{ \bar{q}_-(\eta) }\big(p_{en}'(\eta)  - \frac{m-\overline{m}}{\sigma \overline{m}}\frac{g}{\bar{q}_-(\eta)}\big)\dif \eta\notag\\
 & -\frac{1}{2} \big( \int_0^{\overline{m}}K(\eta)\frac{\partial_{\xi}^2 \delta\theta_-^{(0)}  (\hat{\xi},\eta)}{\sigma}
\dif \eta - (K_2(\eta) - A_+(\eta) )\Theta''(\hat{\xi})\big) \xi^2\notag\\
\geq &  \xi\int_0^{\overline{m}}\frac{K(\eta)}{ \bar{q}_-(\eta) }\big(p_{en}'(\eta)  - \frac{m-\overline{m}}{\sigma \overline{m}}\frac{g}{\bar{q}_-(\eta)}\big)\dif \eta\notag\\
& - \frac{1}{2} \xi \big( \int_0^{\overline{m}}|K(\eta)|\big|\frac{\partial_{\xi}^2 \delta\theta_-^{(0)} (\hat{\xi},\eta)}{\sigma}\big|
\dif \eta+ \|K_2(\eta) - A_+(\eta)\|_{L^\infty[0,\overline{m}]} \|\Theta''\|_{L^\infty[0,L]} \big) \xi\notag\\
 > & \xi\int_0^{\overline{m}}\frac{K(\eta)}{ \bar{q}_-(\eta) }\big(p_{en}'(\eta)  - \frac{m-\overline{m}}{\sigma \overline{m}}\frac{g}{\bar{q}_-(\eta)}\big)\dif \eta\notag\\
& - \frac12 \xi \big(C_{\overline{U}_-} \int_0^{\overline{m}}|K(\eta)|
\dif \eta+ \|  K_2(\eta) - A_+(\eta)\|_{L^\infty[0,\overline{m}]} \|\Theta''\|_{L^\infty[0,L]} \big) \xi,
  \end{align}
where we use $\big|\frac{\partial_{\xi}^2 \delta\theta_-^{(0)} (\hat{\xi},\eta)}{\sigma}\big| < C_{\overline{U}_-}$(see \eqref{dotU-es}) and $\hat{\xi}\in(0,\xi)$.
Employing \eqref{L0lessmin},
it follows that
\begin{align}
J_1'(\xi) >0, \quad \text{for}\quad \xi\in(0,L_0).
\end{align}
Thus,
\vspace{-0.3cm}
\begin{align*}
&\min \limits_{\xi\in(0,L_0)} J_1(\xi) = J_1(0) = A_+(\eta) \int_{0}^L \Theta(\tau) \dif \tau,\\
&\max\limits_{\xi\in(0,L_0)} J_1(\xi) = J_1(L_0) = \int_{0}^{L_0} J_1'(\xi)\dif \xi + J_1(0)\notag\\
=&\int_0^{L_0}{\xi} \int_0^{\overline{m}}\frac{K(\eta)}{ \bar{q}_-(\eta) }\big(p_{en}'(\eta)  - \frac{m-\overline{m}}{\sigma\overline{m}}\frac{g}{\bar{q}_-(\eta)}\big)\dif \eta\dif \xi + J_1(0)\notag\\
 & \quad -\int_0^{L_0} \frac{\xi^2}{2} \big(\int_0^{\overline{m}}K(\eta)\frac{\partial_{\xi}^2 \delta\theta_-^{(0)} (\hat{\xi},\eta)}{\sigma}
\dif \eta - \big(  K_2(\eta) - A_+(\eta)\big)\Theta''(\hat{\xi}) \big)\dif \xi \notag\\
\geq& J_1(0) + \widehat{J}_1(L_0).
\end{align*}
Therefore, if \eqref{eq:<J2<} holds, there exists a unique $\mathcal{K}_0\in (0,L_0)$ such that \eqref{eq2:J1K0=J2} holds.
\end{proof}
In Lemma \ref{determinemathcalK0}, \eqref{assump1}-\eqref{eq:<J2<} are sufficient conditions for the determination of the approximate shock position $\mathcal{K}_0$. If $\int_0^{\overline{m}}\frac{K(\eta)}{ \bar{q}_-(\eta) }\big(p_{en}'(\eta)-\frac{m-\overline{m}}{\sigma \overline{m}}\frac{g}{\bar{q}_-(\eta)}\big)\dif \eta < 0$, we can also determine the approximate position $\mathcal{K}_0$, as the following lemma shows.
\begin{lem}\label{2suffxi0=}
Assume
\begin{align}
&\int_0^{\overline{m}}\frac{K(\eta)}{ \bar{q}_-(\eta) }\big(p_{en}'(\eta)  - \frac{m-\overline{m}}{\sigma \overline{m}}\frac{g}{\bar{q}_-(\eta)}\big)\dif \eta < 0,\label{assump1rem}\\
&0<  L_0^{\sharp}< \min\{L_*^{\sharp},\,\,  L \},\label{L0lessrem}
\end{align}
where $L_*^{\sharp} \defs \frac{-2}{\mathcal{I}^{\sharp}}\int_0^{\overline{m}}\frac{K(\eta)}{ \bar{q}_-(\eta) }\big(p_{en}'(\eta)  - \frac{m-\overline{m}}{\sigma \overline{m}}\frac{g}{\bar{q}_-(\eta)}\big)\dif \eta$ and
\begin{align*}
&\mathcal{I}^{\sharp}\defs C_{\overline{U}_-}\int_0^{\overline{m}}|K(\eta)|
\dif \eta+ \|K_2(\eta) - A_+(\eta)\|_{L^\infty[0,\overline{m}]} \|\Theta''\|_{L^\infty[0,L]}.
\end{align*}
If
\begin{align}\label{<J2<eq}
J_1(0) + \widehat{J}_1^{\sharp}(L_0)< J_2 < J_1(0),
\end{align}
where $  J_1(0) \defs  A_+(\eta) \int_{0}^L \Theta(\tau) \dif \tau$ and
\begin{align*}
  \widehat{J}_1^{\sharp}(L_0^{\sharp})\defs &  \frac{(L_0^{\sharp})^2}{6} \Big( 3\int_0^{\overline{m}}\frac{K(\eta)}{ \bar{q}_-(\eta) }\big( p_{en}'(\eta)  - \frac{m-\overline{m}}{\sigma \overline{m}}\frac{g}{\bar{q}_-(\eta)}\big)\dif \eta + L_0^{\sharp} \mathcal{I}^{\sharp}\Big)<0,
\end{align*}
there exists a unique $\mathcal{K}_0\in (0,L_0^{\sharp})$ such that \eqref{eq2:J1K0=J2} holds.
\end{lem}

\subsection{Nonlinear iteration scheme}
In this subsection, we will design an iteration scheme based on the solution $(\overline{U}_+(\eta) +\delta U_+^{(0)}(\xi,\eta); (\psi^{(0)})'(\eta), \psi^{(0)}(\overline{m}))$. Define
\begin{equation*}
\begin{array}{rcl}
& & \mathscr{N} (\delta U_+^{(0)}; (\psi^{(0)})', \mathcal{K}_0)\\
&\defs &
\begin{Bmatrix}
( \delta U_+; \psi', \psi(\overline{m}) ):\qquad \eqref{solvabilityhat} \,\, \text{holds}, \qquad |\psi(\overline{m}) - \mathcal{K}_0| \leq C\sigma,\\
\| (\delta p_+ , \delta \theta_+) \|_{1,\alpha;\Omega_+^{\mathcal{K}_0}}^{(-\alpha;\{Q_i\})} + \| (\delta q_+ , \delta S_+) \|_{1,\alpha;\Omega_+^{\mathcal{K}_0}}^{(-\alpha;\{\Gamma_0^{\mathcal{K}_0}\cup \Gamma_{\overline{m}}^{\mathcal{K}_0}\})} + \| \psi'\|_{1,\alpha;\Gamma_s^{\mathcal{K}_0}}^{(-\alpha;\{Q_1, Q_4\})}  \leq C\sigma,\\
\| (\delta p_+ - \delta p_+^{(0)}, \delta \theta_+ - \delta \theta_+^{(0)}) \|_{1,\alpha;\Omega_+^{\mathcal{K}_0}}^{(-\alpha;\{Q_i\})} +  \| (\delta q_+ - \delta q_+^{(0)}, \delta S_+ - \delta S_+^{(0)}) \|_{1,\alpha;\Omega_+^{\mathcal{K}_0}}^{(-\alpha;\{\Gamma_0^{\mathcal{K}_0}\cup \Gamma_{\overline{m}}^{\mathcal{K}_0}\})}\\
+ \| \psi' -(\psi^{(0)})'\|_{C^{1,\alpha}(\Gamma_s^{\mathcal{K}_0})}^{(-\alpha;\{Q_1, Q_4\})} \leq \sigma^{\frac32}
\end{Bmatrix}.
\end{array}
\end{equation*}

It is easy to see that $(\delta U_+^{(0)}; (\psi^{(0)})', \mathcal{K}_0) \in  \mathscr{N} (\delta U_+^{(0)}; (\psi^{(0)})', \mathcal{K}_0)$.

For any $(\delta U_+; \delta \psi',\psi(\overline{m}))\in \mathscr{N} (\delta U_+^{(0)}; (\psi^{(0)})', \mathcal{K}_0)$, define $U(\xi,\eta)=\overline{U}_+(\eta)  + \delta U_+(\xi,\eta)$ and $\psi'(\eta) = \delta\psi'(\eta)$. Let $(\delta {U}_+^*; \delta {\psi^*}', \psi^*(\overline{m}))$ be the solutions of equations \eqref{Rrefeq1}-\eqref{Rrefeq4} with boundary conditions \eqref{As-1H}-\eqref{eq:theta4+re}, where $F_k, (k=1,2,3)$, $\mathcal{H}_j^{\sharp}, (j=1,2,3,4)$ are the functions of $(\delta U_+; \delta \psi', \psi(\overline{m}))$ given by \eqref{defsF1}-\eqref{Lf32=} and \eqref{2.73}-\eqref{alpha4+}. Therefore, we construct a mapping
\begin{align}\label{Pimapping}
  \mathcal{T} : (\delta U_+; \delta \psi',\psi(\overline{m})) \mapsto (\delta {U}_+^*; \delta {\psi^*}'. \psi^*(\overline{m}))
\end{align}
We will show the mapping $\mathcal{T}$ is well defined and contractive on $\mathscr{N} (\delta U_+^{(0)}; (\psi^{(0)})', \mathcal{K}_0)$.


The well-defined of $\mathcal{T}$ is achieved in the following lemma.
\begin{lem}\label{WellP}
  There exists a positive constant $\sigma_1$ with $0 < \sigma_1<1$, such that for any $0 < \sigma \leq \sigma_1$, if $( \delta U_+; \psi', \psi(\overline{m}) )\in \mathscr{N}(\delta U_+^{(0)}; (\psi^{(0)})', \mathcal{K}_0)$, then $(\delta U_+^*; (\psi^*)', \psi^*(\overline{m}) )\in \mathscr{N}(\delta U_+^{(0)}; (\psi^{(0)})', \mathcal{K}_0)$.
\end{lem}
\begin{proof}
Because $( \delta U_+; \psi', \psi(\overline{m}) )\in \mathscr{N}(\delta U_+^{(0)}; (\psi^{(0)})', \mathcal{K}_0)$,
it is easy to see that \eqref{Assumpf}-\eqref{assumebry} hold. So by Theorem \ref{Rthm}, if  \eqref{solvabilityhat} holds, we have $(\delta U_+^*; (\psi^*)')$ exists and satisfies
\begin{align}\label{ptheta*es}
&\|(\delta p_+^*,\delta \theta_+^*) \|_{1,\alpha;\Omega_+^{\mathcal{K}_0}}^{(-\alpha;\{Q_i\})} +  \|(\delta q_+^*, \delta S_+^*) \|_{1,\alpha;\Omega_+^{\mathcal{K}_0}}^{(-\alpha;\{\Gamma_0^{\mathcal{K}_0}\cup \Gamma_{\overline{m}}^{\mathcal{K}_0}\})} + \| (\psi^*)'\|_{1,\alpha;\Gamma_s^{\mathcal{K}_0}}^{(-\alpha;\{Q_1,Q_4\})}\notag\\
\leq & C\Big(\sum_{k=1}^3 \|F_k \|_{0,\alpha;\Omega_+^{\mathcal{K}_0}}^{(1-\alpha;\{\Gamma_0^{\mathcal{K}_0}\cup \Gamma_{\overline{m}}^{\mathcal{K}_0}\})}  + \sigma \|p_{ex}\|_{1,\alpha;\Gamma_{ex}}^{(-\alpha;\{Q_2, Q_3\})}\notag\\
 &\qquad + \sigma \|\Theta\|_{C^{2,\alpha}(\Gamma_{\overline{m}}^{\mathcal{K}_0})} +\sum_{j=1}^4 \|h_j\|_{1,\alpha;\Gamma_s^{\mathcal{K}_0}}^{(-\alpha;\{Q_1, Q_4\})} \Big),
\end{align}
where the constant $C$ depends on $\overline{U}_{\pm}$, $L$ and $\alpha$.

Moreover, it is easy to see that for sufficiently small constant $\sigma>0$,
\begin{align}\label{eq013}
&\|\delta p_+^* - \delta p_+^{(0)}\|_{1,\alpha;\Omega_+^{\mathcal{K}_0}}^{(-\alpha;\{Q_i\})}  +  \|\delta \theta_+^* - \delta \theta_+^{(0)}\|_{1,\alpha;\Omega_+^{\mathcal{K}_0}}^{(-\alpha;\{Q_i\})} +  \|\delta q_+^* - \delta q_+^{(0)}\|_{1,\alpha;\Omega_+^{\mathcal{K}_0}}^{(-\alpha;\{\Gamma_0^{\mathcal{K}_0}\cup \Gamma_{\overline{m}}^{\mathcal{K}_0}\})}\notag\\
&+  \|\delta S_+^*-\delta S_+^{(0)}\|_{1,\alpha;\Omega_+^{\mathcal{K}_0}}^{(-\alpha;\{\Gamma_0^{\mathcal{K}_0}\cup \Gamma_{\overline{m}}^{\mathcal{K}_0}\})} +\| (\psi^*)' - (\psi^{(0)})' \|_{1,\alpha;\Gamma_s^{\mathcal{K}_0}}^{(-\alpha;\{Q_1,Q_4\})}\notag\\
  \leq & C \Big(\| F_1 \|_{0,\alpha;\Omega_+^{\mathcal{K}_0}}^{(1-\alpha;\{\Gamma_0^{\mathcal{K}_0}\cup \Gamma_{\overline{m}}^{\mathcal{K}_0}\})}  + \| F_2 -\frac{g}{A_+(\eta)}\frac{m-\overline{m}}{\overline{m}}
  \frac{1}{\bar{q}_+(\eta)} \|_{0,\alpha;\Omega_+^{\mathcal{K}_0}}^{(1-\alpha;\{\Gamma_0^{\mathcal{K}_0}\cup \Gamma_{\overline{m}}^{\mathcal{K}_0}\})}\notag \\
  & \qquad +\|F_3 +B(\overline{U}_+)\|_{0,\alpha;\Omega_+^{\mathcal{K}_0}}^{(1-\alpha;\{\Gamma_0^{\mathcal{K}_0}\cup \Gamma_{\overline{m}}^{\mathcal{K}_0}\})}+ \sum_{j=1}^{4}\| h_j - h_j^{(0)}\|_{1,\alpha;\Gamma_s^{\mathcal{K}_0}}^{(-\alpha;\{Q_1, Q_4\})}\notag \\
  &\qquad +\sigma \| p_{ex}(y(L,\eta;\overline{U}_+ + \delta U_+)) -p_{ex}(y(L,\eta;\overline{U}_+))\|_{1,\alpha;{\Gamma}_{ex}}^{(-\alpha;\{Q_2, Q_3\})} \notag\\
  & \qquad + \sigma \| \Theta\big( \frac{L(\psi(\overline{m}) - \mathcal{K}_0) + (L-\psi(\overline{m}))\xi }{L-\mathcal{K}_0} \big) - \Theta(\xi,\eta)\|_{C^{2,\alpha}(\Gamma_{\overline{m}}^{\mathcal{K}_0})}  \Big)\notag\\
  \leq & C\sigma^2 < \sigma^{\frac32}.
  \end{align}

Next, $\psi^*(\overline{m})$ will be determined from \eqref{solvabilityhat}.
Let
\begin{align}\label{defsJ*=}
  &J_*(\delta {U}_+^*, (\psi^*)',\psi^*(\overline{m}); U_-)\notag\\
  \defs &\int_0^{\overline{m}}\frac{A_+(\eta)}{\bar{\rho}_+(\eta) \bar{q}_+(\eta)} h_1(\mathcal{K}_0, \eta)\dif \eta + A_+(\eta) \int_{\mathcal{K}_0}^L\sigma \Theta \big( \frac{L(\psi^*(\overline{m}) - \mathcal{K}_0) + (L-\psi^*(\overline{m}))\xi }{L-\mathcal{K}_0} \big)\dif \xi\notag\\
 &-\int_0^{\overline{m}}
 \frac{1-\overline{M}_+^2(\eta)}{\bar{\rho}_+^2(\eta) \bar{q}_+^3(\eta)} A_+(\eta) \sigma p_{ex}
   \big(\frac{m}{\overline{m}}\int_{0}^{\eta}\frac{1}{\rho_+ q_+\cos\theta_+ (L,t)} \dif t\big) \dif\eta\notag\\
&+\int_{0}^{\overline{m}}\int_{\mathcal{K}_0}^{L} F_1(\delta U_+^*, (\psi^*)'; \mathcal{K}_0)\dif \xi \dif \eta.
\end{align}
To show \eqref{solvabilityhat} holds,  it suffices to show there exists a unique $\psi^*(\overline{m})$ such that
\begin{align}\label{Solva*=}
  J_*(\delta U_+^*,  (\psi^*)',\psi^*(\overline{m}); U_-) =0,
\end{align}
and
\begin{align}\label{location*}
  |\psi^*(\overline{m}) - \mathcal{K}_0|<C\sigma.
\end{align}
Applying \eqref{dotsolvability}, further calculations yield that
\begin{align}\label{J*=}
  &J_*(\delta {U}_+^*, (\psi^*)',\psi^*(\overline{m}); U_-)\notag\\
=& \int_0^{\overline{m}}\frac{A_+(\eta)}{\bar{\rho}_+(\eta) \bar{q}_+(\eta)} \big(h_1^*(\mathcal{K}_0, \eta) - h_1^{(0)}(\mathcal{K}_0, \eta)\big)\dif \eta\notag\\
& +  A_+(\eta)\sigma \int_{\mathcal{K}_0}^L \Theta \big( \frac{L(\psi^*(\overline{m}) - \mathcal{K}_0) + (L-\psi^*(\overline{m}))\xi }{L-\mathcal{K}_0} \big) - \Theta(\xi)\dif \xi\notag\\
 &-\sigma\int_0^{\overline{m}}
 \frac{1-\overline{M}_+^2(\eta)}{\bar{\rho}_+^2(\eta) \bar{q}_+^3(\eta)} A_+(\eta) \big( p_{ex}
   \big(y(L,\eta;\overline{U}_+ +\delta U_+^*)\big) - p_{ex}(\int_{0}^{\eta}\frac{1}{\bar{\rho}_+(t) \bar{q}_+(t)} \dif t )\big)\dif\eta\notag\\
    &+\int_{0}^{\overline{m}}\int_{\mathcal{K}_0}^{L} F_1(\delta U_+^* , (\psi^*)'; \mathcal{K}_0)\dif \xi \dif \eta.
\end{align}
Obviously,
\begin{align}
J_*(\delta U_+^{(0)}, (\psi^{(0)})',\mathcal{K}_0;\overline{U}_- + U_-^{(0)} ) =0.
\end{align}
 Then we analyze each of the terms in \eqref{J*=}. Firstly, by \eqref{2.73}, \eqref{itera-} and \eqref{eq013}, one has
\begin{align}
&\mathcal{H}_j^{\sharp}(\delta U_+^*,\delta U_-) = \alpha_j^+\cdot \delta U_+^* - \mathcal{H}_j(U_+^*, U_-((\psi^*)', \varphi^*(\overline{m})))\notag\\
   =& \big(\alpha_{j+} \cdot \delta U_+^* +\alpha_{j-} \cdot \delta U_-(\delta \psi',\overline{m}) - G_j (U, U_-((\psi^*)', \varphi^*(\overline{m})))\big)\\
   &-\alpha_{j-} \cdot \big(\delta U_-((\psi^*)', \varphi^*(\overline{m})) - \delta U_-^{(0)}(\mathcal{K}_0,\eta)\big)- \alpha_{j-} \cdot U_-^{(0)}(\mathcal{K}_0,\eta)\notag\\
=&- \alpha_{j-} \cdot \delta U_-^{(0)} (\mathcal{K}_0,\eta) + O(1)\sigma^2.
\end{align}
Further calculations yield that
  \begin{align}
    h_j^*(\mathcal{K}_0, \eta) - h_j^{(0)}(\mathcal{K}_0, \eta) = h_j^{(0)} (\psi^*(\overline{m}),\eta) - h_j^{(0)}(\mathcal{K}_0, \eta) + O(1)\sigma^2.
  \end{align}
Secondly,
  \begin{align}
&\int_{\mathcal{K}_0}^L \Theta \big( \frac{L(\psi^*(\overline{m}) - \mathcal{K}_0) + (L-\psi^*(\overline{m}))\xi }{L-\mathcal{K}_0} \big)- \Theta(\xi)\dif \xi\notag\\
=&\int_{\mathcal{K}_0}^{L}  \Theta(\xi) \dif \xi +\int_{\psi^*(\overline{m})}^{\mathcal{K}_0} \Theta(\xi)\dif \xi + \frac{\psi^*(\overline{m}) - \mathcal{K}_0}{L- \psi^*(\overline{m})}\int_{\psi^*(\overline{m})}^L \Theta(\xi)\dif \xi- \int_{\mathcal{K}_0}^L \Theta(\xi)\dif \xi\notag\\
=& \int_{\psi^*(\overline{m})}^{\mathcal{K}_0} \Theta(\xi)\dif \xi + \frac{\psi^*(\overline{m}) - \mathcal{K}_0}{L- \psi^*(\overline{m})}\int_{\psi^*(\overline{m})}^L \Theta(\xi)\dif \xi,\\
 &p_{ex}
\big(y(L,\eta;\overline{U}_+ +\delta U_+^*)\big) - p_{ex}(\int_{0}^{\eta}\frac{1}{\bar{\rho}_+(t) \bar{q}_+(t)} \dif t )\notag\\
=&p_{ex}
\big(\frac{m}{\overline{m}}\int_{0}^{\eta}\frac{1}{\rho_+^* q_+^*\cos\theta_+^* (L, t)} \dif t \big) - p_{ex}\big( \int_{0}^{\eta}\frac{1}{\bar{\rho}_+(t) \bar{q}_+(t)} \dif t  \big)\notag\\
   =& \int_0^1 p_{ex}' \Big(s \big(\frac{m}{\overline{m}}\int_{0}^{\eta}\frac{1}{\rho_+^* q_+^*\cos\theta_+^* (L, t)} \dif t\big)  + (1-s)  \big( \int_{0}^{\eta}\frac{1}{\bar{\rho}_+(t) \bar{q}_+(t)} \dif t\big) \Big) \dif s\notag\\
   & \times \Big(\frac{m-\overline{m}}{\overline{m}}\int_{0}^{\eta}
   \frac{1}{\rho_+^* q_+^* \cos\theta_+^* (L, t)} \dif t + \int_0^{\eta}\big( \frac{1}{\rho_+^* q_+^*\cos\theta_+^* (L,t)} -  \frac{1}{\bar{\rho}_+(t) \bar{q}_+(t)} \big)  \dif t \Big).
  \end{align}
Thirdly, by \eqref{defsF1}, one has
  \begin{align}
    &F_1 (\delta U_+^*, (\psi^*)'; \mathcal{K}_0)\notag\\
    =& - \frac{\psi^*(\overline{m})-\mathcal{K}_0 - \int_{\eta}^{\overline{m}} (\psi^*)'(t)\dif t}{L -\psi^*(\eta)}A_+(\eta)\partial_\xi \big( \frac{1}{\rho_+^* q_+^* \cos\theta_+^*}\big)\notag\\
    & - \frac{1}{\bar{\rho}_+(\eta)\bar{q}_+^3(\eta)}
    \frac{\psi^*(\overline{m}) - \mathcal{K}_0 -  \int_{\eta}^{\overline{m}} (\psi^*)'(t)\dif t}{L -\psi(\eta)}A_+(\eta)\partial_{\xi} (\frac12 (q_+^*)^2 + i_+^* ) + O(1)\sigma^2\notag\\
    =& - \frac{\psi^*(\overline{m})-\mathcal{K}_0}{L -\psi^*(\eta)}A_+(\eta)\partial_\xi \big( \frac{1}{\rho_+^* q_+^* \cos\theta_+^*}\big) \notag\\
    & - \frac{1}{\bar{\rho}_+(\eta)\bar{q}_+^3(\eta)}
    \frac{\psi^*(\overline{m}) - \mathcal{K}_0 }{L -\psi^*(\eta)}A_+(\eta)\partial_{\xi} (\frac12 (q_+^*)^2 + i_+^*) + O(1)\sigma^2\notag\\
    =&  - \frac{\psi^*(\overline{m})-\mathcal{K}_0}{L -\psi^*(\overline{m})} A_+(\eta)\Big( \partial_\xi \big(-\frac{1}{\bar{\rho}_+^2(\eta)\bar{q}_+(\eta)\bar{c}_+^2(\eta)}
    \delta{p}_+^{(0)} - \frac{1}{\bar{\rho}_+(\eta)\bar{q}_+^2(\eta)}\delta{q}_+^{(0)} \big) \notag\\
    & \qquad \qquad +  \frac{1}{\bar{\rho}_+(\eta)\bar{q}_+^3(\eta)}\partial_{\xi} \big( \frac{1}{\bar{\rho}_+(\eta)}\delta{p}_+^{(0)} + \bar{q}_+ \delta{q}_+^{(0)} + \bar{T}_+(\eta) \delta{S}_+^{(0)}\big)\Big)\notag\\
    & + O(1)\sigma^{\frac32} +O(1)\sigma^2\notag\\
    =&- \frac{\psi^*(\overline{m})-\mathcal{K}_0}{L -\psi^*(\overline{m})}A_+(\eta) \big( \partial_{\eta}\delta{\theta}_+^{(0)} - \frac{g}{\bar{\rho}_+(\eta)\bar{q}_+^3(\eta)}\delta{\theta}_+^{(0)}
    \big) + O(1)\sigma^{\frac32} +O(1)\sigma^2\notag\\
    =& - \frac{\psi^*(\overline{m})-\mathcal{K}_0}{L -\psi^*(\overline{m})}  \partial_{\eta}\big(A_+(\eta) \delta{\theta}_+^{(0)}\big) + O(1)\sigma^{\frac32} +O(1)\sigma^2,
  \end{align}
 where we use $A_+(\eta) =  e^{-g\int_0^\eta
\frac{1}{\bar{\rho}_+(s)\bar{q}_+^3(s)}\dif s}$.
 Thus,
 \begin{align}
 &\int_{0}^{\overline{m}}\int_{\mathcal{K}_0}^{L} F_1(\delta U_+^*, (\psi^*)'; \mathcal{K}_0)\dif \xi \dif \eta \notag\\
 =& - \frac{\psi^*(\overline{m})- \mathcal{K}_0}{L -\psi^*(\overline{m})}  \int_{0}^{\overline{m}}\int_{\mathcal{K}_0}^{L}\partial_{\eta}\big(A_+(\eta) \delta{\theta}_+^{(0)}\big)\dif \xi \dif \eta  + O(1)\sigma^{\frac32} +O(1)\sigma^2\notag\\
 =&- \frac{\psi^*(\overline{m})-{\xi}_0}{L -\psi^*(\overline{m})}  A_+(\eta) \sigma \int_{\mathcal{K}_0}^{L}\Theta(\xi) \dif \xi + O(1)\sigma^{\frac32} +O(1)\sigma^2.
 \end{align}
Therefore, \eqref{J*=} yields that
\begin{align}\label{J*==}
  &J_*(\delta {U}_+^*, (\psi^*)',\psi^*(\overline{m});U_-)\notag\\
  =& \int_0^{\overline{m}}\frac{A_+(\eta)}{\bar{\rho}_+(\eta) \bar{q}_+(\eta)} \partial_{\xi}h_1^{(0)}(\mathcal{K}_0, \eta) \dif \eta\cdot  \big( \psi^*(\overline{m}) - \mathcal{K}_0\big)\notag\\
  & +  A_+(\eta)\sigma \big( \int_{\psi^*(\overline{m})}^{\mathcal{K}_0} \Theta(\xi)\dif \xi + \frac{\psi^*(\overline{m}) - \mathcal{K}_0}{L- \psi^*(\overline{m})}\int_{\psi^*(\overline{m})}^L \Theta(\xi)\dif \xi \big)\notag\\
  & - \frac{\psi^*(\overline{m})- \mathcal{K}_0}{L -\psi^*(\overline{m})}  A_+(\eta) \sigma \int_{\mathcal{K}_0}^{L} \Theta(\xi) \dif \xi + O(1)\sigma^{\frac32} +O(1)\sigma^2\notag\\
  =&\int_0^{\overline{m}}\frac{A_+(\eta)}{\bar{\rho}_+(\eta) \bar{q}_+(\eta)} \partial_{\xi}h_1^{(0)}(\mathcal{K}_0, \eta) \dif \eta\cdot  \big( \psi^*(\overline{m}) - \mathcal{K}_0\big)\notag\\
  & +  A_+(\eta)\sigma \big( \int_{\psi^*(\overline{m})}^{\mathcal{K}_0} \Theta(\xi)\dif \xi + \frac{\psi^*(\overline{m}) - \mathcal{K}_0}{L- \psi^*(\overline{m})}\int_{\psi^*(\overline{m})}^{\mathcal{K}_0} \Theta(\xi)\dif \xi \big)+ O(1)\sigma^{\frac32} +O(1)\sigma^2\notag\\
  =& \int_0^{\overline{m}}\frac{A_+(\eta)}{\bar{\rho}_+(\eta) \bar{q}_+(\eta)} \partial_{\xi}h_1^{(0)}(\mathcal{K}_0, \eta) \dif \eta\cdot  \big( \psi^*(\overline{m}) - \mathcal{K}_0\big)\notag\\
  & -  A_+(\eta)\sigma \Theta(\mathcal{K}_0) \big( \psi^*(\overline{m}) - \mathcal{K}_0  \big)\big(1+  \big( \psi^*(\overline{m}) - \mathcal{K}_0  \big)\big)
  + O(1)\sigma^{\frac32} +O(1)\sigma^2\notag\\
  =& \big(\int_0^{\overline{m}}\frac{A_+(\eta)}{\bar{\rho}_+(\eta) \bar{q}_+(\eta)} \partial_{\xi}h_1^{(0)}(\mathcal{K}_0, \eta) \dif \eta - \sigma A_+(\eta) \Theta(\mathcal{K}_0) \big)\cdot  \big( \psi^*(\overline{m}) - \mathcal{K}_0\big)\notag\\
  & - A_+(\eta)\sigma \Theta(\mathcal{K}_0) \big( \psi^*(\overline{m}) - \mathcal{K}_0  \big)^2 + O(1)\sigma^{\frac32} +O(1)\sigma^2\notag\\
  =& \sigma J_1'(\mathcal{K}_0) \big( \psi^*(\overline{m}) - \mathcal{K}_0\big)+ O(1)\sigma \big( \psi^*(\overline{m}) - \mathcal{K}_0\big)^2 + O(1)\sigma^{\frac32} +O(1)\sigma^2,
  \end{align}
  see \eqref{J1'=} for the definition of $J_1'$.
By Lemma \ref{determinemathcalK0}, \eqref{J*==} implies that
 \begin{align}
   \partial_{\psi^*(\overline{m})} J_*(\delta U_+^{(0)}, (\psi^{(0)})',\mathcal{K}_0;\overline{U}_- + U_-^{(0)} ) =  \sigma J_1^{'} (\mathcal{K}_0) \neq 0.
  \end{align}
Thus, by the implicit function theorem, for sufficiently small positive $\sigma >0$, there exists a unique $\psi^*(\overline{m})$ such that \eqref{Solva*=} and \eqref{location*} hold.
\end{proof}

Theorem \ref{mainthm3} can be proven easily by the Banach fixed point theorem if the mapping $\mathcal{T}$ is contractive on $\mathscr{N}$. It is the following lemma.
\begin{lem}\label{Contra}
  There exists a positive constant $\sigma_2$ with $0<\sigma_2 \ll 1$, such that for any $0< \sigma \leq \sigma_2$, the mapping $\mathcal{T}$ defined on $\mathscr{N} (\delta U_+^{(0)}; (\psi^{(0)})', \mathcal{K}_0)$ is contractive.
 \end{lem}

\begin{proof}
Suppose that $(\delta U_+^{(k)}; (\psi^{(k)})', \psi^{(k)}(\overline{m}))\in \mathscr{N} (\delta U_+^{(0)}; (\psi^{(0)})', \mathcal{K}_0)$, $(k=1,2)$.
By Lemma \ref{WellP}, $((\delta U_+^*)^{(k)};  (\psi^{*(k)})', \psi^{*(k)}(\overline{m}))\in \mathscr{N} (\delta U_+^{(0)}; (\psi^{(0)})', \mathcal{K}_0)$.
Applying \eqref{Solva*=}, direct calculations imply that
\begin{align*}
&|\psi^{*(2)} (\overline{m}) - \psi^{*(1)}(\overline{m}) | \notag\\
\leq& C \big(\| (\delta U_+^*)^{(2)} - (\delta U_+^*)^{(1)} \|_{1,\alpha;\Gamma_s^{\mathcal{K}_0}}^{(-\alpha;\{Q_1,Q_4\})} + \|(\psi^{*(2)})' - (\psi^{*(1)})' \|_{1,\alpha;\Gamma_s^{\mathcal{K}_0}}^{(-\alpha;\{Q_1,Q_4\})}\big),
\end{align*}
where the constant $C$ depends on $\overline{U}_{\pm}$, $L$ and $\alpha$.
Then for sufficiently small $\sigma>0$, it holds that
\begin{align*}
&\|((\delta p_+^*)^{(2)} -  ((\delta p_+^*)^{(1)}\|_{1,\alpha;\Omega_+^{\mathcal{K}_0}}^{(-\alpha;\{Q_i\})} +\|((\delta \theta_+^*)^{(2)} -  ((\delta \theta_+^*)^{(1)}\|_{1,\alpha;\Omega_+^{\mathcal{K}_0}}^{(-\alpha;\{Q_i\}}  \notag\\
& +\|((\delta q_+^*)^{(2)} -  ((\delta q_+^*)^{(1)}\|_{1,\alpha;\Omega_+^{\mathcal{K}_0}}^{(-\alpha;\{\Gamma_0^{\mathcal{K}_0}\cup \Gamma_{\overline{m}}^{\mathcal{K}_0}\})} +\|((\delta S_+^*)^{(2)} -  ((\delta S_+^*)^{(1)}\|_{1,\alpha;\Omega_+^{\mathcal{K}_0}}^{(-\alpha;\{\Gamma_0^{\mathcal{K}_0}\cup \Gamma_{\overline{m}}^{\mathcal{K}_0}\})} \notag\\
&+
\|(\delta \psi_{(2)}^*)' - (\delta \psi_{(1)}^*)'\|_{1,\alpha;\Gamma_s^{\mathcal{K}_0}}^{(-\alpha;\{Q_1,Q_4\})}
+ |\psi_{(2)}^*(\overline{m}) - \psi_{(1)}^*(\overline{m})|\notag \\
\leq& \frac12\Big(\|\delta p_+^{(2)}  - \delta p_+^{(1)}\|_{1,\alpha;\Omega_+^{\mathcal{K}_0}}^{(-\alpha;\{Q_i\})}  + \|\delta \theta_+^{(2)}  - \delta \theta_+^{(1)}\|_{1,\alpha;\Omega_+^{\mathcal{K}_0}}^{(-\alpha;\{Q_i\})}  \notag\\
&\quad + \|\delta q_+^{(2)}  - \delta q_+^{(1)}\|_{1,\alpha;\Omega_+^{\mathcal{K}_0}}^{(-\alpha;\{\Gamma_0^{\mathcal{K}_0}\cup \Gamma_{\overline{m}}^{\mathcal{K}_0}\})} + \|\delta S_+^{(2)}  - \delta S_+^{(1)}\|_{1,\alpha;\Omega_+^{\mathcal{K}_0}}^{(-\alpha;\{\Gamma_0^{\mathcal{K}_0}\cup \Gamma_{\overline{m}}^{\mathcal{K}_0}\})} \notag\\
&\quad + \|  \delta \psi_{(2)}' - \delta \psi_{(1)}' \|_{1,\alpha;\Gamma_s^{\mathcal{K}_0}}^{(-\alpha;\{Q_1,Q_4\})}+ |\psi_{(2)}(\overline{m}) - \psi_{(1)}(\overline{m})|\Big).
\end{align*}
So the mapping is contractive.
\end{proof}

\begin{proof}[Proof of Theorem \ref{mainthm3}]
Because the mapping $\mathcal{T}$ is well defined and contractive on $\mathscr{N} (\delta U_+^{(0)}; (\psi^{(0)})', \mathcal{K}_0)$, there exists a fixed point $( \delta U_+; \psi', \psi(\overline{m}) )\in \mathscr{N}(\delta U_+^{(0)}; (\psi^{(0)})', \mathcal{K}_0)$ of $\mathcal{T}$. Let $\psi=\psi(\overline{m})+\int_0^{\eta}\psi'(\eta)d\eta$. So it is easy to see that $( \delta U_+; \psi)$ is a solution in the sense of Definition \ref{RSP}. With the definition of $\mathscr{N} (\delta U_+^{(0)}; (\psi^{(0)})', \mathcal{K}_0)$ and the unique existence of the supersonic solution $U_-$ ahead of the shock front can be easily established by employing the method of characteristic line (cf. Section 2 - Section 3 of the book \cite{LY1985}),Theorem \ref{mainthm3} follows.
\end{proof}

{\bf Acknowledgments.}
The research of Beixiang Fang was supported in part by Natural Science
Foundation of China under Grant Nos. 12331008 and 11971308, and the Fundamental Research Funds for the Central Universities. The research of Xin Gao was supported in part by
Scholarship of Shanghai Jiao Tong University.
The research of Wei Xiang was supported in part by the
Research Grants Council of the HKSAR, China (Project No. CityU 11304820, CityU 11300021, CityU 11311722 and CityU 11305523).
The research of Qin Zhao was supported in part by Natural Science Foundation of China under Grant Nos. 12101471 and 12272297.

\end{document}